%% file: InterlacementsRevised.tex
\documentclass[12pt, reqno]{amsart}
\title{
\centerline{
Interlacements and the Wired Uniform Spanning Forest
}
}
\author{Tom Hutchcroft}
\address{University of British Columbia}
\email{thutch@math.ubc.ca}
\date{\today}

\input{header.tex}

\input{commands.tex}

\usepackage[numeric,initials,nobysame]{amsrefs}

\usepackage{commath}
\usepackage{mathrsfs}
\usepackage{bbm}
 \usepackage[margin=1.05in]{geometry}
\usepackage{enumitem}

\newcommand{\eps}{\varepsilon}
\newcommand{\AB}{\mathsf{AB}}
\renewcommand{\Cap}{\mathrm{Cap}}
\newcommand{\OUST}{\mathsf{OUST}}
\newcommand{\past}{\mathrm{past}}

\begin{document}

\begin{abstract}
We  extend the Aldous-Broder algorithm to generate the wired uniform spanning forests (WUSFs) of infinite, transient graphs. 
We do this by replacing the simple random walk in the classical algorithm with Sznitman's random interlacement process.
We then apply this algorithm to study the WUSF, showing that every component of the WUSF is one-ended almost surely in any graph satisfying a certain weak anchored isoperimetric condition, that the number of
  `excessive ends' in the WUSF is non-random in any graph, and also  
that every component of the WUSF is one-ended almost surely in any transient unimodular random rooted graph. The first two of these results answer positively two questions of Lyons, Morris and Schramm [\emph{Electron. J. Probab.} \textbf{13} (2008), no. 58, 1702\ndash 1725], while the third extends a recent result of the author. 

Finally, we construct a counterexample showing that almost sure one-endedness of WUSF components is not preserved by rough isometries of the underlying graph, answering negatively a further question of Lyons, Morris and Schramm.  
\end{abstract}






\maketitle

\section{Introduction}
The \textbf{uniform spanning forests} (USFs) of an infinite, locally finite, connected graph $G$ are defined as weak limits of uniform spanning trees (USTs) of large finite subgraphs of $G$. These weak limits can be taken with either free or wired boundary conditions (see \cref{Sec:bgUSF}), yielding the \textbf{free uniform spanning forest} (FUSF) and \textbf{wired uniform spanning forest} (WUSF) respectively. The USFs are closely related to several other topics in probability theory, including loop-erased random walks \cite{Lawler80,Wilson96}, potential theory \cite{BurPe93,BLPS}, conformally invariant scaling limits \cite{Schramm00,LaSchWe04},  domino tiling \cite{Ken00} and the Abelian sandpile model \cite{Dhar90,Jar14}.
In this paper, we develop a new connection between the wired uniform spanning forest and Sznitman's \textbf{interlacement process} \cite{Szni10,Teix09}.

A key theoretical tool in the study of the UST and USFs is \textbf{Wilson's algorithm}~\cite{Wilson96}, which allows us to sample the UST of a finite graph by joining together loop-erasures of random walk paths. In their seminal work \cite{BLPS}, 
Benjamini, Lyons, Peres, and Schramm  (henceforth referred to as BLPS) extended Wilson's algorithm to infinite transient graphs and used this extension to establish several fundamental properties of the WUSF. For example, they proved that the WUSF of an infinite, locally finite, connected graph is connected 
 almost surely (a.s.) if and only if the sets of vertices visited by two independent random walks on the graph have infinite intersection a.s. 
This recovered the earlier, pioneering work of Pemantle~\cite{Pem91}, who proved that the FUSF and WUSF of  $\Z^d$ coincide for all $d$ and are a.s.\ connected  if and only if $d\leq4$.
Wilson's algorithm has also been instrumental in the study of scaling limits of uniform spanning trees and forests \cite{LaSchWe04,Schramm00,peres2004scaling,schweinsberg2009loop,barlow2014subsequential}.

 Prior to the introduction of Wilson's algorithm, the best known algorithm for sampling the UST of a finite graph was the \textbf{Aldous-Broder algorithm} \cite{Aldous90,broder1989generating}, which generates a uniform spanning tree of a finite connected graph $G$ as the collection of first-entry edges of a random walk on $G$. We now describe this algorithm in detail.
Let $\rho$ be a fixed vertex of $G$, and let $\langle X_n \rangle_{n\geq0}$ be a simple random walk on $G$ started at $\rho$. For each vertex $v$ of $G$, let $e(v)$ be the edge of $G$ incident to $v$ that is traversed by the random walk $X_n$ as it enters $v$ for the first time, and let 
$T=\left\{e(v) : v \in V \setminus \{\rho\}\right\}$
be set of first-entry edges. 
Aldous \cite{Aldous90} and Broder~\cite{broder1989generating} proved independently that the resulting random spanning tree $T$ is distributed uniformly on the set of spanning trees of $G$ (see also \cite[\S 4.4]{LP:book}). 
If we orient the edge in the direction opposite to that in which it was traversed by the random walk, then the spanning tree is oriented towards $\rho$, meaning that every vertex of $G$ other than $\rho$ has exactly one oriented edge emanating from it in the tree.

While the algorithm 
 extends without modification to generate USTs of recurrent infinite graphs, 
the collection of first entry edges of a random walk on a  transient graph might not span the graph. Thus, naively running the Aldous-Broder on a transient graph will not necessarily produce a spanning forest of the graph. 
Moreover, unlike in Wilson's algorithm, we cannot simply continue the algorithm by starting another random walk from a new location. 
As such, it has hitherto been unclear how to extend the Aldous-Broder algorithm to infinite transient graphs and, as a result, the Aldous-Broder algorithm has been of limited theoretical use in the study of USFs of infinite graphs. 

In this paper, we extend the Aldous-Broder algorithm to infinite, transient graphs by replacing the random walk with the random interlacement process.
The interlacement process was originally introduced by Sznitman \cite{Szni10} to study the disconnection of cylinders and tori by a random walk trajectory, and was generalised to arbitrary transient graphs by Teixeira~\cite{Teix09}. 
The interlacement process $\sI$ on a transient graph $G$ is a point process on the space $\cW^\ast \times \R$, where $\cW^\ast$ is the space of doubly-infinite paths in $G$ modulo time-shift (see \cref{Sec:bginterlacement} for precise definitions), and should be thought of as a collection of random walk excursions from infinity. 
We refer the reader to the monographs \cite{DrRaBa14} and \cite{CerTei12} for an introduction to the extensive literature on the random interlacement process.

We state our results in the natural generality of networks. Recall that a \textbf{network} $(G,c)$ is a connected, locally finite graph $G=(V,E)$, possibly containing self-loops and multiple edges, together with a function $c:E\to(0,\infty)$ assigning a positive \textbf{conductance} $c(e)$ to each edge $e$ of $G$. The conductance $c(v)$ of a vertex $v$ is defined to be the sum of the conductances of the edges emanating from $v$. 
Graphs without specified conductances are considered as networks by setting $c(e)\equiv1$. We will usually suppress the notation of conductances, and write simply $G$ for a network. See \cref{Sec:bgUSF} for detailed definitions of the USFs on general networks. 

Oriented edges $e$ are oriented from their tail $e^-$ to their head $e^+$. The reversal of an oriented edge $e$ is denoted $-e$.

\begin{thm}[Interlacement Aldous-Broder]\label{thm:IAB}
Let $G$ be a transient, connected, locally finite network, let $\sI$ be the interlacement process on $G$, and let $t\in \R$. For each vertex $v$ of $G$, let $\tau_t(v)$ be the smallest time greater than $t$ such that there exists a trajectory $(W_{\tau_t(v)},\tau_t(v))\in \sI$ passing through $v$, and let $e_t(v)$ be the oriented edge of $G$ that is traversed by the trajectory $W_{\tau_t(v)}$ as it enters $v$ for the first time. Then 
\[\AB_t(\sI):=\left\{-e_t(v):v \in V\right\}\]
has the law of the oriented wired uniform spanning forest of $G$.
\end{thm}


A useful feature of the interlacement Aldous-Broder algorithm is that it allows us to consider the wired uniform spanning forest of an infinite transient graph as the stationary  measure of the ergodic Markov process $\langle \AB_t(\sI) \rangle_{t\in \R}$. Indeed, it is with this stationarity in mind that we consider the interlacement process to be a point process on $\cW^\ast \times \R$ rather than the more usual $\cW^\ast \times \R_+$. For example, a key step in proving that the number of excessive ends of the WUSF is non-random is to show that the number of \emph{indestructible} excessive ends is a.s.\ monotone in the time evolution of the process $\langle \AB_t(\sI) \rangle_{t\in \R}$. 

\section{Applications}

\subsection{Ends}
Other than connectivity, the most basic topological property of a forest is the number of ends its components have. Here, an infinite, connected graph $G=(V,E)$ is said to be \textbf{$k$-ended} if, over all finite subsets $W$ of $V$, the subgraph of $G$ induced by $V\setminus W$ has a supremum of $k$ infinite connected components. In particular, an infinite tree is $k$-ended if and only if there exist exactly $k$ distinct infinite simple paths starting at each vertex of the tree. Components of the WUSF are known to be one-ended a.s.\ in several large classes of graphs. The first result of this kind is due to Pemantle \cite{Pem91}, who proved that the WUSF of $\Z^d$ is one-ended a.s.\ for $2\leq d \leq 4$, and that every component of the WUSF of $\Z^d$ has at most two ends a.s.\ for every $d\geq 5$ (the WUSF of $\Z$ is the whole of $\Z$ and is therefore two-ended). BLPS \cite{BLPS} later completed this work, showing in particular that every component of the WUSF is one-ended  a.s.\ in any transient Cayley graph. We note that one-endedness of WUSF components has important consequences for the Abelian sandpile model \cite{JarRed08,JarWer14,Jar14}.

Taking a different approach, Lyons, Morris and Schramm \cite{LMS08} gave an isoperimetric criterion for one-endedness of WUSF components, from which they deduced that the every component of the WUSF is one-ended in every transitive graph not rough isometric to $\Z$, and also every non-amenable graph. Unlike the earlier results of BLPS, the results of Lyons, Morris and Schramm are robust in the sense that their assumptions depend only upon the coarse geometry of the graph and do not require any kind of homogeneity. 
They  asked \cite[Question 7.9]{LMS08} whether the isoperimetric assumption in their theorem could be replaced by the \textbf{anchored} version of the same condition, and in particular whether every WUSF component is one-ended a.s.\ in any graph with anchored expansion (defined below).  
Unlike classical isoperimetric conditions, anchored isoperimetric conditions are often preserved under random perturbations such as supercritical Bernoulli percolation~\cite{chen2004anchored,Pete08}. 

Given a network $G$ and a set $K$ of vertices of $G$, we write $\partial_E K$ for the set of edges of $G$ with exactly one endpoint in $K$, and write $|K|$ for the sum of the conductances of the vertices in $K$. Similarly, if $W$ is a set of edges in $G$, we write $|W|$ for the sum of the conductances of the edges in $W$.  Given an increasing function $f:(0,\infty)\to(0,\infty)$, we say that $G$ satisfies an \textbf{anchored $f(t)$-isoperimetric inequality} if 
\[\inf \left\{\frac{|\partial_E K|}{f\left(|K|\right)}:\;  K\subset V \text{ connected, } v\in K,\,   |K| < \infty\right\}>0\]
for every vertex $v$ of $G$. (In contrast, the graph is said to satisfy a (non-anchored) $f(t)$-isoperimetric inequality if the infimum  $\inf|\partial_E K|/f(|K|)$ is positive when  taken over \emph{all}  sets of vertices $K$ with $|K|<\infty$.) In particular, $G$ is said to have \textbf{anchored expansion} if and only if it satisfies an anchored $t$-isoperimetric inequality, and is said to satisfy a \textbf{$d$-dimensional anchored isoperimetric inequality} if it satisfies an anchored $t^{(d-1)/d}$-isoperimetric inequality. Such anchored isoperimetric inequalities are known to hold on, for example, supercritical
percolation clusters on $\Z^d$
and related graphs, such as half-spaces and wedges \cite{Pete08}.



\begin{thm}\label{T:anchor}
Let $G$ be a network with $c_0:=\inf_e c(e)>0$, and suppose that $G$ satisfies an anchored $f(t)$-isoperimetric inequality for some increasing function $f:(0,\infty)\to(0,\infty)$ for which there exists a constant $\alpha$ such that  $f(t) \leq t$ and $f(2t)\leq \alpha f(t)$ for all $t \in (0,\infty)$. Suppose that $f$ also satisfies each of the following conditions:
\begin{enumerate}
\item  \[\int_{c_0}^\infty\! \frac{1}{f(t)^2}\mathrm{d}t<\infty\]
\end{enumerate}
\vspace{.1em}
 and 
 \begin{enumerate}
 \item[2.]
 \[\int_{c_0}^\infty\! \exp\left(-\eps\left(\int_s^\infty \frac{1}{f(t)^2}\mathrm{d}t\right)^{-1}\right)\mathrm{d}s<\infty\]
 for every $\eps>0$.
 \end{enumerate}
 Then every component of the wired uniform spanning forest of $G$ is one-ended almost surely. 
\end{thm}

In particular,  \cref{T:anchor} applies both to every graph with anchored expansion and to every graph satisfying a $d$-dimensional anchored isoperimetric inequality with $d>2$. 
 The
graph formed by joining two copies of $\Z^2$
together with a single edge between their origins
satisfies a 
$2$-dimensional isoperimetric inequality but has a two-ended WUSF.
The theorem can fail if edge conductances are not bounded away from zero, as can be seen by attaching an infinite path with exponentially decaying edge conductances to the root of a $3$-regular tree.

\cref{T:anchor} comes very close to giving a complete answer to  \cite[Question 7.9]{LMS08}. The isoperimetric condition of \cite{LMS08} is essentially that $G$ satisfies an $f(t)$-isoperimetric inequality for some $f$ satisfying all  conditions of \cref{T:anchor} with the possible exception of $(2)$; the precise condition required is slightly weaker than this but also more technical. Our formulation of \cref{T:anchor} is adapted from the presentation of the results of \cite{LMS08} given in \cite[Theorem 10.43]{LP:book}. 
The difference in requirements on the function $f(t)$ between \cref{T:anchor} and \cite[Theorem 10.43]{LP:book} can be seen by considering $f(t)$ of the form $t^{1/2}\log^\alpha(1+t)$:
In particular, we observe that \cite[Theorem 10.43]{LP:book} applies to graphs satisfying a $t^{1/2}\log^\alpha(1+t)$-isoperimetric inequality for some $\alpha>1/2$, while our theorem applies to graphs satisfying an \emph{anchored} $t^{1/2}\log^\alpha(1+t)$-isoperimetric inequality only if $\alpha>1$.


In \cref{Sec:Rough}, we give an example of two bounded degree, rough-isometric graphs $G$ and $G'$ such that every component of the WUSF of $G$ is one-ended, while the WUSF of $G'$ a.s.\ contains a component with uncountably many ends. This answers negatively Question 7.6 of \cite{LMS08}, and shows that the behaviour of the WUSF of a graph cannot always be determined from the coarse geometric properties of the graph alone.

\subsection{Excessive ends}

One example of a transient graph in which the WUSF has multiply-ended components is the subgraph of $\Z^6$ spanned by the vertex set
\[\left(\Z^5\times\{0\}\right)\cup\left(\left\{(0,0,0,0,0),(2,0,0,0,0)\right\}\times \N \right),\]
which is obtained from $\Z^5$ by attaching an infinite path to each of the vertices $u=(0,0,0,0,0)$ and $v=(2,0,0,0,0)$. The WUSF of this graph, which we denote  $\F$, is equal in distribution to the union of the WUSF of $\Z^5$ with each of the two added paths. If $u$ and $v$ are in the same component of $\F$, then there is a single component of $\F$ with three ends and all other components are one-ended. Otherwise, $u$ and $v$ are in different components of $\F$, so that there are exactly two components of $\F$ that are two-ended and all other components are one-ended. Each of these events has positive probability, so that the event that there exists a two-ended component of the WUSF has probability strictly between $0$ and $1$. Nevertheless, the \textbf{number of excessive ends} of $\F$, that is, the sum over all components of $\F$ of the number of ends minus 1, is equal to two a.s.

In light of this example, Lyons, Morris and Schramm \cite[Question 7.8]{LMS08} asked whether the number of excessive ends of the WUSF is non-random (i.e., equal to some constant a.s.) for any graph. Our next application of the interlacement Aldous-Broder algorithm is to answer this question positively. 

\begin{thm}\label{Thm:excessnotrandom} Let $G$ be a network. Then the number of excessive ends of the wired uniform spanning forest of $G$ is non-random. \end{thm}

When combined with the spatial Markov property of the wired uniform spanning forest, \cref{Thm:excessnotrandom} has the following immediate corollary, which states that a natural weakening of \cite[Question 7.6]{LMS08} has a positive answer. (As mentioned above, we show the original question to have a negative answer in \cref{Sec:Rough}).

\begin{corollary}\label{cor:finmod}
Let $G$ be a network, and suppose that $G'$ is a network obtained from $G$ by adding and deleting finitely many edges. Then the wired uniform spanning forests of $G$ and $G'$ have the same number of excessive ends almost surely. In particular, if every tree of the wired uniform spanning forest of $G$ is one-ended a.s., then the same is true of $G'$.
\end{corollary}

\subsection{Ends in unimodular random rooted graphs}

Another generalisation of the one-endedness theorem of BLPS \cite{BLPS} concerns transient \textbf{unimodular random rooted graphs}. 
A \textbf{rooted graph} is a connected, locally finite graph $G$ together with a distinguished vertex $\rho$, the \textbf{root}. An isomorphism of graphs is an isomorphism of rooted graphs if it preserves the root.  The \textbf{local topology} on the space $\cG_\bullet$ of isomorphism classes of rooted graphs is defined so that two rooted graphs are close if they have large graph distance balls around their respective roots that are isomorphic as rooted graphs. Similarly, a doubly rooted graph is a graph together with an ordered pair of distinguished vertices, and the local topology on the space $\cG_{\bullet\bullet}$ of isomorphism classes of doubly rooted graphs is defined similarly to the local topology on $\cG_\bullet$. A random rooted graph $(G,\rho)$ is said to be \textbf{unimodular} if it satisfies the \textbf{Mass-Transport Principle}, which states that for every Borel function $f:\cG_{\bullet\bullet}\to[0,\infty]$ (which we call a \textbf{mass transport}),
\[\E\bigg[\sum_{v\in V}f(G,\rho,v)\bigg]=\E\bigg[\sum_{v\in V}f(G,v,\rho)\bigg].\]
In other words, the expected mass sent by the root is equal to the expected mass received by the root. Unimodular random rooted networks are defined similarly by allowing the mass-transport to depend on the edge conductances. 
We refer the reader to  Aldous and Lyons~\cite{AL07} for a systematic development and overview of the theory of unimodular random rooted graphs and networks, as well as several examples.

Aldous and Lyons \cite{AL07} proved that every component of the WUSF is one-ended a.s.\ in any bounded degree unimodular random rooted graph, and the author of this article  \cite{H15} later extended this to all transient unimodular random rooted graphs with finite expected degree at the root, deducing that every component of the WUSF is one-ended a.s.\ in any supercritical Galton Watson tree. (The assumption of finite expected degree was implicit in \cite{H15} since there we considered \emph{reversible} random rooted graphs, which correspond to unimodular random rooted graphs with finite expected degree.)
Our final application of the interlacement Aldous-Broder algorithm is to extend the main result of \cite{H15} by removing the condition that the expected degree of the root is finite.

\begin{thm}\label{thm:unimod} Let $(G,\rho)$ be a transient unimodular random rooted network. Then every component of the wired uniform spanning forest of $G$ is one-ended almost surely. \end{thm}

\section{Background and definitions}

\subsection{Uniform Spanning Forests} \label{Sec:bgUSF}
For each finite graph $G=(V,E)$, let $\UST_G$ denote the uniform measure on the set of spanning trees of $G$ (i.e. connected, cycle-free subgraphs of $G$ containing every vertex), which are considered for measure-theoretic purposes to be functions $E \to \{0,1\}$. More generally, for each finite network $G$, let $\UST_G$ denote the probability measure on the set of spanning trees of $G$ such that the probability of a tree is proportional to the product of the conductances of its edges.

 Let $G$ be an infinite network, and let $\langle V_n\rangle_{n\geq 0}$ be an exhaustion of $V$ by finite connected subsets, i.e.~ an increasing sequence of finite connected subsets of $V$ such that $\bigcup V_n = V$. For each $n$, let $G_n$ denote the subnetwork of $G$ induced by $V_n$, and let $G_n^*$ denote the finite network obtained by identifying (wiring) every vertex of $G$ in $V\setminus V_n$ into a single vertex $\partial_n$, and deleting the infinitely many self-loops from $\partial_n$ to itself. The \textbf{wired uniform spanning forest} measure is defined as the weak limit of the uniform spanning tree measures on $G_n^*$. That is, for every finite set $S \subset E$,
\[ \WUSF_G(S \subseteq \F) := \lim_{n\to\infty}\UST_{G_n^*}(S \subseteq T).\]
In contrast, the \textbf{free uniform spanning forest} measure is defined as the weak limit of the uniform spanning tree measures on the finite induced subnetworks $G_n$. It is easily seen that both the free and wired measures are supported on the set of \textbf{essential spanning forests} of $G$, that is, the set of cycle-free subgraphs of $G$ that contain every vertex and do not have any finite connected components. 

It will also be useful to consider oriented trees and forests. Given an infinite network $G$ with exhaustion $\langle V_n \rangle_{n\geq0}$, let $\OUST_{G_n^*}$ denote the law of a uniform spanning tree of $G_n^*$ that has been oriented towards the boundary vertex $\partial_n$, meaning that every vertex of $G_n^*$ other than $\partial_n$ has exactly one oriented edge emanating from it in the tree.  BLPS \cite{BLPS} proved that if $G$ is transient, then the measures $\OUST_{G_n^*}$ converge weakly to a measure $\OWUSF$, the \textbf{oriented wired uniform spanning forest} (OWUSF) measure. This measure is supported on the set of  essential spanning forests of $G$ that are oriented so that every vertex of $G$ has exactly one oriented edge emanating from it in the forest. The WUSF of a transient graph can be obtained from the OWUSF of the graph by forgetting the orientations of the edges. 

\subsection{The space of trajectories} \label{Sec:bgtrajectories}
Let $G$ be a graph. 
For each $-\infty \leq n \leq m \leq \infty$, let $L(n,m)$ be the graph with vertex set $\{i \in \Z : n \leq i\leq m\}$ and with edge set $\{(i,i+1): n\leq i \leq m-1\}$. We define $\cW(n,m)$ to be the set of multigraph homomorphisms from $L(n,m)$ to $G$ such that the preimage of each vertex in $G$ is finite, 
and define $\cW$ to be the union
\[ \cW:= \bigcup\left\{\cW(n,m): -\infty \leq n \leq m \leq \infty\right\}.\]
For each set $K \subseteq V$, we let $\cW_K(n,m)$ denote the set of $w\in \cW(n,m)$ that visit $K$ (that is, for which there exists $n\leq i\leq m$ such that $w(i)\in K$), and let $\cW_K$ be the union $\cW_K=\bigcup \{\cW_K(n,m): -\infty \leq n \leq m \leq \infty\}$.

Given $w\in \cW(n,m)$ and $a\leq b\in\Z$, we define  $w|_{[a,b]} \in \cW(n\vee a, m \wedge b)$ to be the restriction of $w$ to the subgraph $L(n\vee a, m \wedge b)$ of $L(n,m)$.
Given $w\in \cW_K(n,m)$, let $H^-_K(w)=\inf\{ n \leq i \leq m: w(i)\in K\}$, let $H^+_K(w)=\sup\{ n \leq i \leq m: w(i)\in K\}$, and let \[w_K=w|_{[H^-_k(w),H^+_K(w)]}\] be the restriction of $w$ to between the first it visits $K$ and last time it visits $K$.  
We equip $\cW$ with the topology generated by open sets of the form
\[\{w\in \cW : w \text{ visits K and } w_K= w'_K\},\]
where $K\subset V$ is finite and $w'\in \cW_K$. (Note that this topology is not the weakest topology making the evaluation maps $w \mapsto w(i)$ and $w \mapsto w(i,i+1)$ continuous. First and last hitting times of finite sets are not continuous with respect to that topology, but are continuous with respect to ours. The Borel $\sigma$-algebras generated by the two topologies are the same.)
We also equip $\cW$ with the Borel $\sigma$-algebra generated by this topology. 

The \textbf{time shift}  $\theta_k:\cW\to \cW$ is defined by
$\theta_k : \cW(n,m) \longrightarrow \cW(n-k,m-k)$,
\[ \theta_k(w)(i)=w(i+k), \quad  \theta_k(w)(i,i+1)= w(i+k,i+k+1).\]
The space $\cW^*$ is defined to be the quotient 
\[\cW^* = \cW / \sim \text{ where } w_1\sim w_2 \text{ if and only if } w_1 = \theta_k (w_2) \text{ for some k}.\]
Let $\pi : \cW \to \cW^*$ denote the quotient map. $\cW^*$ is equipped with the quotient topology and associated quotient $\sigma$-algebra. 
 An element of $\cW^*$ is called a \textbf{trajectory}.

\subsection{The interlacement process} \label{Sec:bginterlacement}
 Given a network $G=(G,c)$, the conductance $c(v)$ of a vertex $v$ is defined to be the sum of the conductances of the edges emanating from $v$, and, for each pair of vertices $(u,v)$, the conductance $c(u,v)$ is defined to be the sum of the conductances of the edges connecting $u$ to $v$.  
Recall that the \textbf{random walk} $X$ on the network $G$ is the Markov chain on $V$ with transition probabilities $p(u,v) = c(u,v)/c(u)$. In case $G$ has multiple edges, we will also keep track of the edges crossed by $X$, considering $X$ to be a random element of $\cW(0,\infty)$. We write either $P_u$ or $P^G_u$ for the law of the random walk started at $u$ on the network $G$, depending on whether the network under consideration is clear from context.  
 When $X$ is a random walk on a network $G$ and $K$ is a set of vertices in $G$, we let $\tau_K$ denote the first time that $X$ hits $K$ and let $\tau_K^+$ denote the first positive time that $X$ hits $K$.

Let $G=(V,E)$ be a transient network. Given $w \in \cW(n,m)$, let $w^\leftarrow \in \cW(-n,-m)$ be the reversal of $w$, defined by setting $w^\leftarrow(i)=w(-i)$ for all $-m \leq i \leq -n$, and $w^\leftarrow(i,i+1)=w(-i-1,-i)$ for all $-m \leq i \leq -n-1$. For each subset $\sA \subseteq \cW$, let $\sA^\leftarrow$ denote the set 
\[\sA^\leftarrow:= \{w \in \cW: w^\leftarrow \in \sA \}.\]
For each finite set $K\subset V$, define a measure $Q_K$ on $\cW_K$ by setting \[Q_K(\{w\in \cW: w(0)\notin K\})=0\] and, for each $u\in K$ and each two Borel subsets $\sA,\sB\in \cW$,  
\begin{multline*}
Q_K\left( \{w \in \cW: w|_{(-\infty,0]} \in \sA,\, w(0) = u \text{ and } w|_{[0,\infty)} \in \sB \}\right)\\= 
 c(u)P_u\big( \langle X_{k} \rangle_{k\geq0}\in \sA^\leftarrow \text{ and } \tau_K^+ =\infty\big) P_u\big(\langle X_k \rangle_{k\geq0} \in \sB \big). 
\end{multline*}
Let $\cW^*_K= \pi(W_K)$ be the set of trajectories that visit $K$.
\begin{thm}[Sznitman \cite{Szni10} and Teixeira \cite{Teix09}:  Existence and uniqueness of the interlacement intensity measure] Let $G$ be a transient network. There exists a unique $\sigma$-finite measure $Q^\ast$ on $\cW^*$ such that for every Borel set $\sA \subseteq \cW^*$ and every finite $K\subset V$,
\begin{equation}\label{eq:ildefn} Q^*(\sA \cap \cW^*_K) =  Q_K\left(\pi^{-1}(\sA)\right). \end{equation}
\end{thm} The measure $Q^*$ is referred to as the  \textbf{interlacement intensity measure}.

\begin{defn}
Let $\Lambda$ denote the Lebesgue measure on $\R$. 
The \textbf{interlacement process} $\sI$ on $G$ is defined to be a Poisson point process on $\cW^* \times \R$ with intensity measure $Q^* \otimes \Lambda$. For each $t\in\R$, we denote by $\sI_t$ the set of $w\in \cW^*$ such that $(w,t)\in\sI$. We also write $\sI_{[a,b]}$ for the intersection of $\sI$ with $\cW^* \times [a,b]$. 
\end{defn}

Let $\langle V_n \rangle_{n\geq0}$ be an exhaustion of an infinite transient network $G$. 
The interlacement process on $G$ can be constructed as a limit of Poisson processes on random walk excursions from the boundary vertices $\partial_n$ to itself in the networks $G_n^*$. 

 Let $N$ be a Poisson point process on $\R$ with intensity measure $c(\partial_n)\Lambda$. Conditional on $N$, for every $t\in N$, let $W_t$ be a random walk started at $\partial_n$ and stopped when it first returns to $\partial_n$, where we consider each $W_t$ to be an element of $\cW^*$. We define $\sI^n$ to be the point process
\[\sI^n:=\left\{(W_t,t) : t \in N\right\}.\]

\begin{prop}\label{Prop:intexhaust}
Let $G$ be an infinite transient network and let $\langle V_n \rangle_{n\geq 0}$ be an exhaustion of $G$. Then the Poisson point processes $\sI^n$ converge in distribution to the interlacement process $\sI$  as $n\to\infty$.
\end{prop}

A similar construction of the random interlacement process is sketched in \cite[\S 4.5]{Szni12book}. 

\begin{proof}
Let $K \subset V$ be finite, and let $n$ be sufficiently large that $K$ is contained in $V_n$. Define a measure $Q^n_K$ on $\cW$ by setting
\[Q^n_K\left(\{w\in\cW: w(0)\notin K\}\right)=0\]
and, for each $u\in K$, each $r,m\geq 0$ and each two Borel subsets $\sA,\sB\in \cW$, 
\begin{multline}\label{eq:Qkndefn}
Q_K^n\left( \left\{w \in \cW(-r,m): w|_{[-r,0]} \in \sA, w(0) = u \text{ and } w|_{[0,m]} \in \sB \right\}\right)\\= 
 c(u)P^{G_n^*}_u\big( \langle X_k \rangle_{k=0}^r\in\sA^\leftarrow \text{ and } \tau_K^+ > \tau_{\partial_n} =r \big) P^{G_n^*}_u\big(\langle X_k \rangle_{k=0}^m \in \sB  \text{ and } \tau_{\partial_n} = m\big).
\end{multline}
By reversibility of the random walk, the right-hand side of \eqref{eq:Qkndefn} is equal to
\begin{multline}
  c(\partial_n)P^{G_n^*}_{\partial_n}\big( \langle X_k \rangle_{k=0}^n\in\sA \text{ and } X_{\tau_K}=u \text{ and } \tau_K=r < \tau_{\partial_n}^+\big)\\ \cdot P^{G_n^*}_u\big(\langle X_k \rangle_{k=0}^m \in \sB  \text{ and } \tau_{V \setminus V_n} = m\big).\end{multline}
  It follows that $Q_K^n(\cW)= c(\partial_n)P^{G_n^*}_{\partial_n}(\tau_K<\tau^+_{\partial_n})$ and that the normalized measure $Q_K^n/Q_K^n(\cW)$ coincides with the law of a random walk excursion from $\partial_n$  to itself in $G_n^*$ that has been conditioned to hit $K$ and reparameterised so that it first hits $K$ at time $0$.
Thus, by the splitting property of Poisson processes, 
 $\sI^n$ is a Poisson point process on $\cW^*\times\R$ with intensity measure $Q^{n*}\otimes\Lambda$, where $Q^{n*}$ satisfies
\[ Q^{n*}(\sA \cap \cW^*_K) =  Q^n_K\left(\pi^{-1}(\sA)\right). \]
We conclude the proof by noting that $Q_K^n$ converges weakly to $Q_K$ as $n\to\infty$.
\end{proof}

\subsection{Hitting Probabilities}\label{Sec:bghitting}
Recall that the \textbf{capacity} (a.k.a.\ the \textbf{conductance to infinity}) of a finite set of vertices $K$ in a network $G$ is defined to be 
\[ \Cap(K) = \sum_{v\in K} c(v)P_v(\tau_K^+ = \infty),\]
and observe that 
 $Q_K(\cW)=\Cap(K)$ for every finite set of vertices $K$. If $K$ is infinite, we define 
 $\Cap(K) = \lim_{n\to\infty} \Cap(K_n)$,  
 where $K_n$ is any increasing sequence of finite sets of vertices with $\bigcup K_n = K$.  
We say that a set $K$ of vertices is \textbf{hit} by $\sI_{[a,b]}$ if there exists $(W,t)\in\sI_{[a,b]}$ such that $W$ hits $K$. By the definition of $\sI$, we have that
\[ \P(K \text{ hit by } \sI_{[a,b]}) = 1-\exp\left(-(b-a)Q_K(\cW)\right) = 1-\exp\left(-(b-a)\Cap(K)\right) \]
for each finite set $K$. This formula extends to infinite sets by taking limits over exhaustions: If $K \subseteq V$ is infinite, let $K_n$ be an exhaustion of $K$ by finite sets. Then
\begin{align*}\P(K \text{ hit by } \sI_{[a,b]}) &= \lim_{n\to\infty}\P(K_n \text{ hit by } \sI_{[a,b]})\\&=1-\lim_{n\to\infty}\exp\left(-(b-a)\Cap(K_n)\right) = 1-\exp\left(-(b-a)\Cap(K)\right).\end{align*}
Similarly, the expected number of trajectories in $\sI_{[a,b]}$ that hit $K$ is equal to $(b-a)\Cap(K)$. 
We apply the above formulas to deduce the following simple 0-1 law.
\begin{lem}\label{L:zeroonehit} Let $G$ be a transient network, let $\sI$ be the interlacement process on $G$. Then for all $a<b \in \R$ and every set of vertices $K\subseteq V$, we have
\[\P\left(\text{$K$} \textrm{ is hit by infinitely many trajectories in } \sI_{[a,b]}\right) = \mathbbm{1}\left(\Cap(K)=\infty\right).\] 
\end{lem}
\begin{proof}
If $\Cap(K)$ is finite then the expected number of trajectories in $\sI_{[a,b]}$ that hit $K$ is finite, so that the number of trajectories in $\sI_{[a,b]}$ that hit $K$ is finite a.s.
Conversely, if $\Cap(K)$ is infinite, then there is a trajectory in $\sI_{[b-2^{-n},b-2^{-n-1}]}$ that hits $K$ a.s.\ for every $n\geq 1$ a.s. Since $b-2^{-n}\geq a$ for all but finitely many $n$, it follows that  $\sI_{[a,b]}$ hits $K$ infinitely often a.s. 
\end{proof}

We next prove that any set that has a positive probability to be hit infinitely often by any single trajectory will in fact be hit by infinitely many trajectories. 
Recall the \textbf{method of random paths} \cite[Theorem 10.1]{PeresClimb}: If $G$ is an infinite network, $A$ is a finite subset of $G$ and $\Gamma$ is a random infinite simple path in $G$ starting at $A$, then 
\[\Cap(A)^{-1} \leq \sum_{e \in E}\P(e\in \Gamma)^2. \]
In particular, if the sum on the right hand side is finite for some random infinite simple path $\Gamma$ starting at $A$, then the capacity of $A$ is positive and $G$ is therefore transient. Moreover, for every finite set $A$ in a transient network, there exists a random infinite simple path $\Gamma$ starting in $A$ such that 
\[\Cap(A)^{-1} = \sum_{e \in E}\P(e\in \Gamma)^2.\]

The following lemma is presumably well-known, but we were unable to find a reference.

\begin{lem}\label{L:recurrentcap}
Let $G$ be a transient network and let $K\subseteq V$. If the random walk on $G$ hits $K$ infinitely often with positive probability, then $\Cap(K)=\infty$.
\end{lem}

\begin{proof}
Let $c$ be the conductance function of $G$.
First suppose that $K$ is hit infinitely often with probability one. 
Let $X=\langle X_i \rangle_{i\geq0}$ be a random walk on $G$ started at a vertex of $K$, and let $N_i$ be the $i$th time that $X$ visits $K$. 
Define
\[c_K(u,v) = c_K(u)P_u(X_{N_1}=v) \]
for all $u,v \in K$, so that $X_{N_i}$ is the random walk on the (non-locally finite) network $H:=((K,K^2),c_K)$. Note that if $A$ is a finite subset of $K$, then the capacity of $A$ considered as a set of vertices in $H$ is the same as the capacity of $A$ considered as a set of vertices in $G$. 
That is,
\[ \Cap(A) = \sum_{v\in A}c(v)P_v(\text{$\langle X_i \rangle_{i\geq0}$ returns to $A$}) = \sum_{v\in A}c_K(v)P_v(\text{$\langle X_{N_i} \rangle_{i\geq0}$ returns to $A$}). \]
Let $\langle K_n \rangle_{n\geq1}$ be an increasing sequence of finite sets with $\bigcup_{n\geq1} K_n = K$, and let $\Gamma$ be a random infinite simple path in $H$ starting at $K_1$ such that
\[\frac{1}{2}\sum_{u,v\in V} \P(\{u,v\} \in \Gamma)^2 = \Cap(K_1)^{-1} <\infty.\]
For each $n\geq2$, let $\Gamma_n$ be the subpath of $\Gamma$ beginning at the last time $\Gamma$ visits $K_n$. Then, by the monotone convergence theorem,
\[\Cap(K) = \lim_{n\to\infty}\Cap(K_n) \leq \lim_{n\to\infty}\Bigg(\frac{1}{2}\sum_{u,v \in V} \P(\{u,v\} \in \Gamma_n)^2\Bigg)^{-1} = \infty.\]
This concludes the proof in this case.

Now suppose that $K$ is hit infinitely often by the random walk on $G$ with positive probability, and, for each vertex $u$ of $G$, let $h(u)$ be the probability that a random walk on $G$ started at $u$ hits $K$ infinitely often. 
We have that, for each two vertices $u$ and $v$ of $G$, 
\[P_u(X_1=v \mid X \text{ hits $K$ infinitely often}) = \frac{P_u(X_1=v)P_v(X \text{ hits $K$ infinitely often})}{P_u(X \text{ hits $K$ infinitely often})} = \frac{h(v)}{h(u)}c(u,v).  \]
It follows by an elementary calculation that the random walk on $G$ conditioned to hit $K$ infinitely often is reversible, and is equal to the random walk on the network $(G,\hat c)$, where the conductances $\hat c$ are defined by
\[\hat c(u,v) = c(u,v)h(u)h(v), \quad \hat c(u) = c(u) h(u)^2.\]
This is an example of Doob's $h$-transform. Since $h\leq 1$, Rayleigh monotonicity implies that the capacity of $K$ with respect to the $h$-transformed conductances $\hat c$ is less than the capacity of $K$ with respect to the original conductances $c$. Thus, we may conclude the proof by applying the argument of the previous paragraph to the $h$-transformed network. \qedhere
\end{proof}

The following lemma is an immediate consequence of \cref{L:zeroonehit} and \cref{L:recurrentcap}.

\begin{lem}\label{L:zeroonehit2} Let $G$ be a transient network, let $\sI$ be the interlacement process on $G$. Then for all $a<b \in \R$ and every set of vertices $K\subseteq V$, we have
\begin{multline*}
\P\left(\textrm{infinitely many vertices of $K$} \textrm{ are hit by } \sI_{[a,b]}\right) =\\
\P\left(\text{$K$} \textrm{ is hit by infinitely many trajectories in } \sI_{[a,b]}\right) = \mathbbm{1}\left(\Cap(K)=\infty\right).
\end{multline*} 
\end{lem}

\medskip

\section{Interlacement Aldous-Broder}\label{Sec:IAB}
In this section we describe the Interlacement Aldous-Broder algorithm and investigate its basic properties. 
Let $G$ be an infinite transient network. For each set $A\subseteq \cW^*\times\R$, and each vertex $v\in V$, define
\[ \tau_t(A,v):=\inf\left\{s\geq t: \exists (W,s) \in A \text{ such that } W \text{ hits }v \right\}.\]
Let $\sI$ be the interlacement process on $G$ and write $\tau_t(v)=\tau_t(\sI,v)$.
Let $\cA$ be the set of subsets $A$ of $\cW^*\times \R$ that satisfy the property that for every 
  for every vertex $v \in V$ and every $t\in \R$ there exists a unique trajectory $W_{\tau_t(A,v)}$ such that $(W_{\tau_t(A,v)},\tau_t(A,v))\in A$ and $W_{\tau_t(A,v)}$ hits $v$. It is clear that $\sI\in\cA$ a.s. Define $e_t(v)=e_t(\sI,v)$ to be the oriented edge pointing into $v$ that is traversed by the trajectory $W_{\tau_t(v)}$ as it enters $v$ for the first time. 
 For each $t\in\R$ and $T\in (t,\infty]$, we define the set $\mathsf{AB}_t^T(\sI)\subseteq E$ by
\begin{equation}\label{eq:intABdefn}\mathsf{AB}_t^T(\sI) := \left\{ -e_t(v) : v \in V,\, \tau_t(v) \leq T\right\}.\end{equation}
We write $\AB_t(\sI)=\AB_t^\infty(\sI)$. We define $\AB_t^T(A)$ similarly for all $A \in \cA$. Let $\langle V_n \rangle_{n\geq0}$ be an exhaustion of $G$, and for each $n\geq 0$ let $\sI^n$ be defined as in \cref{Sec:bginterlacement}. Since the process $\sI^n$ is just a decomposition of a single random walk trajectory into excursions, 
\cref{thm:IAB} follows immediately from the following lemma together with the correctness of the classical Aldous-Broder algorithm.
\begin{lem} Let $G$ be a transient network with exhaustion $\langle V_n \rangle_{n\geq0}$ and let $\sI$ be the interlacement process on $G$. Then
 $\AB_t^T(\sI^n)$ converges weakly to $\AB_t^T(\sI)$ for each $t\in \R$ and $T\in[t,\infty]$.
\end{lem}

\begin{proof} 
Let $E^\rightarrow$ be the set of oriented edges of $G$, let $S$  be a finite subset of $E^\rightarrow$, and consider the set
\[C_t^T(S):=\left\{A \subseteq \cW^*\times\R : S \subseteq \AB_t^T(A)\right\}.\]
Let $\partial C_t^T(S)$ denote the topological boundary of $C_t^T(S)$. Observe that
\[\partial C_t^T(S) \subseteq \left\{ A \subseteq \cA : \tau_t(e^-) \in\{t,T\} \text{ for some } e\in S\right\} \cup (\cW^*\times \R \setminus \cA).\]
It follows that $\P(\sI \in \partial C_t^T(S))=0$. The Portmanteau Theorem \cite[Theorem 13.16]{Klenkebook} therefore implies that 
\begin{equation*} \P\left(S \subseteq \AB_t^T(\sI)\right)=\P\left(\sI \in C_t^T(S)\right) = \lim_{n\to\infty}\P\left(\sI^n \in C_t^T(S)\right)=\lim_{n\to\infty}\P\left(S \subseteq \AB_t^T(\sI^n)\right)  \end{equation*}
for every finite set $S\subseteq E^\rightarrow$.
\end{proof}

We next establish the basic properties of the process $\langle \F_{t}\rangle_{t\in\R} = \langle \AB^\infty_{t}(\sI) \rangle_{t\in \R}$. We first recall that a process $\langle X_t \rangle_{t\in \R}$ is said to be \textbf{ergodic} if $\P(\langle X_t \rangle_{t\in \R} \in \sA)\in \{0,1\}$ whenever $\sA$ is an event that is shift invariant in the sense that $\langle X_t \rangle_{t\in\R} \in \sA$ implies that $\langle X_{t+s} \rangle_{t\in \R}\in\sA$ for every $s\in \R$. The process $\langle X_t \rangle_{t\in \R}$ is said to be \textbf{mixing} if for every two events $\sA$ and $\sB$,
\[\P\left(\langle X_t \rangle_{t\in\R} \in \sA \text{ and } \langle X_{t+s} \rangle_{t\in\R} \in \sB\right)  \xrightarrow[s\to\infty]{}\P\left(\langle X_t \rangle_{t\in\R} \in \sA\right)\P\left(\langle X_{t} \rangle_{t\in\R} \in \sB\right).\] 
Every mixing process is clearly ergodic, but the converse need not hold in general \cite[\S 20.5]{Klenkebook}.

\begin{prop}\label{Prop:mixing} Let $G$ be a transient network and let $\sI$ be the interlacement process on $G$. 
Then $\langle \F_{t} \rangle_{t\in\R} = \langle\mathsf{AB}_{t}^\infty(\sI)\rangle_{t \in \R}$ is an ergodic, mixing, Markov process.
\end{prop}
\begin{proof}
The fact that $\langle \F_{t}\rangle_{t\in\R}$ is a Markov process follows from the following identity, which immediately implies that  $\langle \F_s \rangle_{s\leq t}$ and $\langle \F_s \rangle_{s\geq t}$ are conditionally independent given $\F_t$ for each $t\in \R$: whenever $t\in \R$ and $s\geq 0$, 
\begin{equation}\label{eq:Markov} \AB_{t-s}(\sI) = \AB_{t-s}^t(\sI) \cup \{ e \in \AB_t(\sI) : e^- \text{ is not hit by } \sI_{[t-s,t)}\}.\end{equation} 

We now prove that $\langle \F_{t}\rangle_{t\in\R}$ is mixing. Let $a_1,a_2,\ldots,a_n$ and $b_1,b_2,\ldots,b_m$ be two increasing sequences of real numbers, and let $A_1,A_2,\ldots,A_n$ and $B_1,B_2,\ldots B_m$ be finite subsets of $E^\rightarrow$. Let $K$ denote the set of endpoints of edges in the union $\bigcup_{i=1}^n A_i$. Let $s$ be sufficiently large that $b_1+s \geq a_n$. Let  $\sA = \{A_i \subseteq \AB_{a_i}^\infty(\sI) \text{ for all }1\leq i\leq n \}$, $\sB_s = \{B_i \subseteq \AB_{b_i+s}^\infty(\sI) \text{ for all } 1\leq i \leq m\}$,
and 
$\sA'_s = \{A_i \subseteq \AB_{a_i}^{b_1+s}(\sI) \text{ for all } 1 \leq i\leq n \}\subseteq \sA$. The events $\sA'_s$ and $\sB_s$ are independent and
\[\P(\sA\setminus\sA'_s) \leq \P(\sI_{[a_n,b_1+s]}\text{ does not hit some vertex in } K) \leq \sum_{v\in K}\exp\left(-(b_1+s-a_n)\mathrm{Cap}(v)\right).\]
We deduce that
\begin{align*}
|\P(\sA\cap \sB_s) - \P(\sA)\P(\sB_s)| &\leq |\P(\sA\cap \sB_s) - \P(\sA'_s\cap \sB_s)| + |\P(\sA'_s\cap \sB_s) - \P(\sA)\P(\sB_s)|\\
&= |\P(\sA\cap \sB_s) - \P(\sA'_s\cap \sB_s)| + |\P(\sA'_s)\P(\sB_s) - \P(\sA)\P(\sB_s)|\\
&\leq 2\sum_{v\in K}\exp\left(-(b_1+s-a_n)\mathrm{Cap}(v)\right) \xrightarrow[s\to\infty]{}0. 
\end{align*} 
Since events of the form $\{A_i \subseteq \AB_{a_i}^\infty(\sI) \text{ for all }1\leq i\leq n \}$ generate the Borel $\sigma$-algebra on the space of $(E^\rightarrow)^{\{0,1\}}$-valued processes, it follows that $\langle \F_{t} \rangle_{t\in \R}$ is mixing and therefore ergodic.
\end{proof}

\section{Proof of \cref{T:anchor,Thm:excessnotrandom,thm:unimod}} \label{Sec:ends}
Recall that an \textbf{end} of a tree is an equivalence class of infinite simple paths in the tree, where two infinite simple paths are equivalent if their traces have finite symmetric difference. Similarly, if $\F$ is a spanning forest of a network $G$, we define an end of $\F$ to be an equivalence class of infinite simple paths in $G$ that eventually only use edges of $\F$, where, again, two infinite simple paths are equivalent if their traces have finite symmetric difference. If an infinite tree $T$ is oriented so that every vertex of the tree has exactly one oriented edge emanating from it, then there is exactly one end $\xi$ of $T$ for which the paths representing $\xi$ eventually follow the orientation of $T$. We call this end the \textbf{primary end} of $T$. We call ends of $T$ that are not the primary end \textbf{excessive ends} of $T$. Note that if $\xi$ is an excessive end of $T$ and $\gamma$ is a simple path representing $\xi$, then all but finitely many of the edges traversed by the path $\gamma$ are traversed in the opposite direction to their orientation in $T$.

If $\sI$ is the interlacement process on a transient network $G$, we write $\cI_{[a,b]}$ for the set of vertices of $G$ hit by $\sI_{[a,b]}$.

The proofs of \cref{T:anchor,thm:unimod} both rely on the following criterion.
See \cref{fig:criterion} for an illustration of the proof. 

\begin{lem}\label{L:criterion}
Let $G$ be a transient network, let $\sI$ be the interlacement process on $G$, and let $\langle \F_{t} \rangle_{t\in \R}=\langle \AB_{t}(\sI)\rangle_{t\in \R}$. If the connected component  containing $v$ of 
$\past_{\F_0}(v) \setminus \cI_{[-\eps,0]}$ 
is finite a.s.\ for every vertex $v$ of $G$ and every $\eps>0$, then every component of $\F_0$ is one-ended a.s.
\end{lem}
\begin{proof}
If $u$ is in the past of $v$ in $\F_{-\eps}$, then $\tau_{-\eps}(u)\geq \tau_{-\eps}(v)$. Thus, on the event that $v$ is not hit by $\sI_{[-\eps,0]}$, the past of $v$ in $\F_{-\eps}$ is equal to the component containing $v$ in the subgraph of $\past_{\F_0}(v)$ induced by the complement of $\cI_{[-\eps,0]}$ (see \cref{fig:criterion}). By assumption, the connected component containing $v$ in this subgraph is finite a.s., and so, by stationarity,  
\begin{multline*}\P(\past_{\F_0}(v) \text{ is infinite})=\P(\past_{\F_{-\eps}}(v) \text{ is infinite}) \leq \P(v\in \cI_{[-\eps,0]})=1-e^{-\eps\Cap(v)}\xrightarrow[\eps\to0]{}0.\end{multline*}
Since $v$ was arbitrary, we deduce that every component of $\F_0$ is one-ended a.s.
\end{proof}

\begin{figure}
\includegraphics[width=0.9\textwidth]{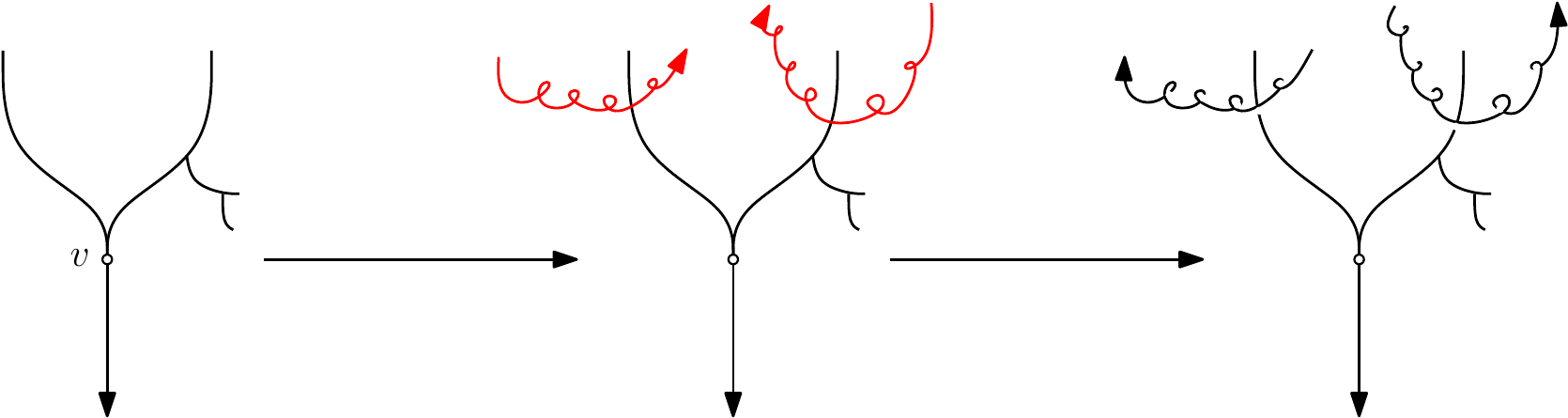}
\caption{A schematic illustration of the proof of \cref{L:criterion}. Left: a tree in $\F_0$ with two excessive ends in the past of the vertex $v$. Centre: each of the excessive ends is hit by a trajectory of $\sI_{[-\eps,0]}$ (red), but $v$ is not hit by such a trajectory. Right: the resulting forest $\F_{-\eps}$. }
\label{fig:criterion}
\end{figure}

The proof of \cref{T:anchor} requires both of the following theorems.
\begin{thm}[Lyons, Morris and Schramm \cite{LMS08}; Lyons and Peres {\cite[Theorem 6.41]{LP:book}}]\label{Thm:LPbound} Let $G$ be a network satisfying an anchored $f(t)$-isoperimetric inequality, where $f$ is an increasing function such that $f(t)\leq t$, $f(2t)\leq \alpha f(t)$ for some constant $\alpha$, and $\int_{1}^\infty f(t)^{-2}\mathrm{d}t<\infty$.
 Then for every vertex $v$ of $G$ there exists a positive constant $c_v$ such that, for every connected set $K$ containing $v$,
 \[\Cap(K) \geq \frac{c_v^2}{4\alpha^2}\left(\int_{|K|}^\infty\!\frac{1}{f(t)^2}\right)^{-1}.\]
 In particular, $G$ is transient. 
\end{thm}

\begin{thm}[Morris {\cite[Theorem 9]{Morris03}}: WUSF components are recurrent]\label{T:recurrent}
Let $G$ be an infinite network with edge conductances bounded above. Then every component of the wired uniform spanning forest of $G$ is recurrent  a.s. 
\end{thm}

An equivalent statement of \cref{T:recurrent} is the following.
\begin{lem}\label{T:recurrentedit}
Let $G$ be an infinite network with $\inf_e c(e) >0$. Then every component of the wired uniform spanning forest of $G$ is recurrent when given unit conductances a.s. 
\end{lem}
\begin{proof}
Form a network $\hat G$ by replacing each edge $e$ of $G$  with $\lceil c(e) \rceil$ parallel edges each with conductance $c(e)/\lceil c(e) \rceil$. It follows immediately from the definition of the UST of a network that the WUSF $\hat \F$ of $\hat G$ may be coupled with the WUSF $\F$ of $G$ so that an edge $e$ of $G$ is contained in $\F$ if and only if one of the edges corresponding to $e$ in $\hat G$ is contained in $\hat\F$. Since $\hat G$ has edge conductances bounded above, \cref{T:recurrent} implies that every component of $\hat \F$ is recurrent a.s. Since the edge conductances of $\hat G$ are bounded away from zero, it follows by Rayleigh monotonicity that every component of $\hat \F$ is recurrent a.s.\ when given unit conductances, and consequently that the same is true of $\F$.  \qedhere
\end{proof}
\begin{proof}[Proof of \cref{T:anchor}]

By \cref{Thm:LPbound}, $G$ is transient. Let $\sI$ be the interlacement process on $G$. 
Let $v$ be a fixed vertex of $G$. Then for each vertex $u$ of $G$ contained in the same component of $\F=\AB_0(\sI)$ as $v$, the conditional probability given $\F$ that $u$ is connected to $v$ in $\F \setminus \cI_{[-\eps,0]}$ is equal to 
$\exp({-\eps\Cap(\gamma_{u,v})})$, 
where $\gamma_{u,v}$ is the trace of the path connecting $u$ to $v$ in $\F$. Write $d_\F(u,v)$ for the graph distance in $\F$ between two vertices $u$ and $v$ in $G$.  
Since the conductances of $G$ are bounded below by some positive constant $\delta$, we have that $|\gamma_{u,v}|\geq \delta d_\F(u,v)$ and so, by \cref{Thm:LPbound},
\begin{multline}\label{eq:percestimate}\P\left(u \text{ connected to } v \text{ in } \F \setminus \cI_{[-\eps,0]} \mid \F\right)\\ \leq \mathbbm{1}(u \text{ connected to } v \text{ in } \F)\exp\left(-\eps\frac{c_v^2}{4\alpha^2}\left(\int_{{\delta}d_\F(u,v)}^\infty\frac{1}{f(t)^2}\right)^{-1}\right).\end{multline}

\begin{lem}\label{L:perc}
Let $T$ be an infinite tree, let $v$ be a vertex of $T$ and let $\omega$ be a random subgraph of $T$. For each vertex $u$ of $G$, let $\Vert u \Vert$ denote the distance between $u$ and $v$ in $T$, and suppose that there exists a function $p:\N\to[0,1]$ such that \[\P(u \text{ is connected to $v$ in $\omega$}) \leq p\left(\Vert u \Vert\right)\] 
for every vertex $u$ in $T$ and
\[\sum_{n\geq1} p(n) < \infty.\]
Then 
\[\P(\text{The component of $v$ in $\omega$ is infinite}) \leq \Cap(v) \sum_{n\geq1}p(n).\]
In particular, if $T$ is recurrent then the component containing $v$ in $\omega$ is finite a.s. 
\end{lem}

\begin{proof}
Suppose that the component containing $v$ in $\omega$ is infinite with positive probability; the inequality holds trivially otherwise. Denote this event $\sA$. Fix a drawing of $T$ in the plane rooted at $v$. On the event $\sA$, let $\Gamma=\Gamma(\omega)$ be the leftmost simple path from $v$ to infinity in $\omega$. 
Observe that \[\P(u \in \Gamma \mid \sA) \leq \frac{\P(u \text{ is connected to $v$ in $\omega$})}{\P(\sA)}\leq \frac{p(n)}{\P(\sA)} \, \text{ and } \,
\sum_{\Vert u\Vert =n}\P(u \in \Gamma \mid \sA) =1,\]
so that
\[\sum_{\Vert u\Vert =n} \P(u \in \Gamma \mid \sA)^2 \leq \frac{p(n)}{\P(\sA)}.\]
and hence 
\begin{multline*}\sum_{e\in E}\P(e \in \Gamma \mid \sA)^2=\sum_{u\in V\setminus\{v\}}\P(u \in \Gamma \mid \sA)^2 = \sum_{n\geq 1}\sum_{\Vert u\Vert =n} \P(u \in \Gamma \mid \sA)^2 \leq \frac{1}{\P(\sA)}\sum_{n\geq1}p(n). \end{multline*}
Applying the method of random paths (taking our measure on random paths to be the conditional distribution of $\Gamma$ given $\sA$), we deduce that
\[\Cap(v) \geq \left(\frac{1}{\P(\sA)}\sum_{n\geq1}p(n)\right)^{-1},\]
which rearranges to give the desired inequality. 
\end{proof}

 Comparing sums with integrals, hypothesis $(2)$ of \cref{T:anchor} implies that
\[\sum_{n\geq1} \exp\left(-\eps\frac{c_v^2}{4\alpha^2}\left(\int_{\delta n}^\infty\frac{1}{f(t)^2}\right)^{-1}\right) < \infty\]
for every $\eps>0$, and we deduce from \cref{eq:percestimate}, \cref{L:perc} and \cref{T:recurrentedit} that the component containing $v$ in $\F \setminus \cI_{[-\eps,0]}$ is finite a.s. Since $v$ was arbitrary, every component of $\F\setminus \cI_{[-\eps,0]}$ is finite a.s. We conclude by applying \cref{L:criterion}.
\end{proof}

\subsection{Unimodular random rooted graphs}\label{Subsec:reversible}

\begin{proof}[Proof of \cref{thm:unimod}]
Let $(G,\rho)$ be a transient unimodular random rooted network, let $\sI$ be the interlacement process on $G$ and let $\F=\AB_0(\sI)$. It is known \cite[Theorem 6.2, Proposition 7.1]{AL07} that every component of $\F$ has at most two ends a.s. Suppose for contradiction that $\F$ contains a two-ended component with positive probability. The \textbf{trunk} of a two-ended component of $\F$ is defined to be the unique doubly infinite simple path that is contained in the component. Define $\text{trunk}(\F)$ to be the set of vertices of $G$ that are contained in the trunk of some two-ended component of $\F$. For each vertex $v$ of $G$, let $e(v)$ be the unique oriented edge of $G$ emanating from $v$ that is contained in $\F$. For each vertex $v\in \text{trunk}(\F)$, let $s(v)$ be the unique vertex in $\text{trunk}(\F)$ that has $e(s(v))^+=v$, and let $s^n(v)$ be defined recursively for $n\geq0$ by $s^0(v)=v$, $s^{n+1}(v)=s(s^n(v))$.

 Let $\eps>0$. We claim that $s^n(\rho)\in \cI_{[-\eps,0]}$ for infinitely many $n$ a.s.\ on the event that $\rho\in \text{trunk}(\F)$.
Let $k\geq1$ and define the mass transport 
\[f_k(G,u,v,\F,\cI_{[-\eps,0]}) =\mathbbm{1}\left(
\begin{array}{l} u \text{ is in the trunk of its component in $\F$},\\ \text{$v=s^k(u)$ and }\cI_{[-\eps,0]}\cap\{s^n(v):n\geq0\}\neq \emptyset \end{array}
\right).\]
Applying the mass-transport principle to $f_k$, we deduce that
\begin{equation*}\P\left(\cI_{[-\eps,0]}\cap\{s^n(\rho):n\geq0\} \neq \emptyset \mid \rho\in \text{trunk}\right) = \P\left(\cI_{[-\eps,0]}\cap\{s^n(\rho):n\geq k\} \neq \emptyset\mid \rho \in \text{trunk}\right)
\end{equation*}
for all $k\geq 0$. (Here we are using the fact that $(G,\rho,\F,\cI_{[-\eps,0]})$ is a unimodular rooted marked graph, see \cite{AL07} for appropriate definitions.) By taking the limit as $k\to\infty$,  we deduce that 
\begin{multline*}\P\left(\text{$s^n(\rho)\in\cI_{[-\eps,0]}$ for infinitely many $n$} \mid \rho\in \text{trunk}\right)\\ = \P\left(\cI_{[-\eps,0]}\cap\{s^n(\rho):n\geq0\} \neq \emptyset \mid \rho\in \text{trunk}\right).
\end{multline*}
It follows that
$s^n(\rho)\in\cI_{[-\eps,0]}$ for infinitely many $n$ almost surely on the event that $\rho\in\text{trunk}(\F)\cap\cI_{[-\eps,0]}$. 
Since $\rho\in\text{trunk}(\F)\cap\cI_{[-\eps,0]}$ with positive probability conditional on $\F$ and the event that $\rho\in\mathrm{trunk}(\F)$, it follows from  \cref{L:zeroonehit2} that infinitely many vertices of $\{s^n(\rho):n\geq1\}$ are hit by $\cI_{[-\eps,0]}$ a.s.\ on the event that $\rho \in\text{trunk}(\F)$.

 It follows from \cite[Lemma 2.3]{AL07} that for every vertex $v\in \text{trunk}(\F)$, $s^n(v)\in\cI_{[-\eps,0]}$ for infinitely many $n$   a.s., and consequently that the component containing $v$ in $\past_\F(v)\setminus \cI_{[-\eps,0]}$ is finite for every vertex $v$ of $G$ a.s. We conclude by applying \cref{L:criterion}. \qedhere

\end{proof}

\subsection{Excessive Ends}
\label{Subsec:Excends}

\begin{proof}[Proof of \cref{Thm:excessnotrandom}]
We may assume that $G$ is transient: if not, the WUSF of $G$ is connected, the number of excessive ends of $G$ is tail measurable, and the claim follows by tail-triviality of the WUSF \cite{BLPS}. 
Let $\sI$ be the interlacement process on $G$ and let $\langle \F_t\rangle_{t\in\R} = \langle \AB_t(\sI)\rangle_{t\in \R}$. 
The event that $\F_0$ has uncountably many ends is tail measurable, and hence has probability either 0 or 1, again by tail-triviality of the WUSF. If the number of ends of $\F_0$ is uncountable a.s., then $\F_0$ must also have uncountably many excessive ends a.s., since the number of components of $\F_0$ is countable.
 Thus, it suffices to consider the case that $\F_0$ has countably many ends a.s.

For each $t\in \R$, we call an excessive end $\xi$ of $\F_t$ \textbf{indestructible} if 
$\Cap\left(\{\gamma_i : i \geq 0\}\right)$ is finite 
for some (and hence every) simple path $\langle \gamma_i \rangle_{i\geq0}$ in $G$ representing $\xi$, and \textbf{destructible} otherwise.
Given a simple path $\gamma=\langle \gamma_i \rangle_{i\geq0}$, write $\gamma_{i,i+1}$ for the oriented edge that is traversed by $\gamma$ as it moves from $\gamma_i$ to $\gamma_{i+1}$, and let $\gamma_{i,i-1}=-\gamma_{i-1,i}$.  
Observe, as we did at the beginning of \cref{Sec:ends}, that a simple path $\langle \gamma_i \rangle_{i\geq0}$ in $G$ represents an excessive end of $\F_t$ if and only if 
$e_t(\gamma_{i},\sI) = \gamma_{i,i-1}$
for all sufficiently large values of $i$ (equivalently, if and only if the reversed oriented edges $-\gamma_{i,i+1}$ are contained in $\F_t$ for all sufficiently large values of $i$). 
 Since $\F_0$ has countably many ends a.s.\ by assumption, it follows from \cref{L:zeroonehit} that for every destructible end $\xi$ of $\F_0$ and every infinite simple path $\langle \gamma_i \rangle_{i\geq0}$ in $G$ representing $\xi$, the trace $\{ \gamma_i : i\geq 0\}$ of $\gamma$ is hit by $\sI_{[-\eps,0]}$ infinitely often a.s.\ for every $\eps>0$.  
Recall from the proof of \cref{L:criterion} that, on the event that $v\notin \cI_{[-\eps,0]}$,  the past of $v$ in $\F_{-\eps}$ is contained in subgraph of $\past_{\F_0}(v)$ induced by the complement of $\cI_{[-\eps,0]}$. It follows that $v$ a.s.\ does not have any destructible ends in its past in $\F_{-\eps}$ a.s.\ on the event that $v\notin \cI_{[-\eps,0]}$, and so, by stationarity,
\[\P(v \text{ has a destructible end in its past in $\F_0$})\leq\P(v \in \cI_{[-\eps,0]}) \xrightarrow[\eps\to0]{} 0.  \] 
Since the vertex $v$ was arbitrary, we deduce that $\F_0$ does not contain any destructible excessive ends a.s.

 Since every excessive end of $\F_0$ is indestructible a.s., it follows from  \cref{L:zeroonehit2} that for every excessive end $\xi$ of $\F_0$ and every path $\langle v_i \rangle$ representing $\xi$, only finitely many of the vertices $v_i$ are hit by $\sI_{[t,0]}$ a.s.\ for every $t\leq 0$. Since $\F_0$ has at most countably many excessive ends a.s., we deduce that 
every path $\langle v_i \rangle_{i\geq 0}$ that represents an excessive end of $\F_0$ also represents an excessive end of $\F_t$ for every $t\leq0$. In particular, the cardinality of the set of excessive ends of $\F_t$ is at least the cardinality of the set of excessive ends of $\F_0$ a.s.\ for every $t\leq 0$. Since, by \cref{Prop:mixing}, $\langle \F_{t}\rangle_{t\in\R}$ is stationary and ergodic, we deduce that the cardinality of the set of excessive ends of $\F_0$ is a.s.\ equal to some constant. 
\end{proof}

\begin{proof}[Proof of \cref{cor:finmod}]
Let $G''$ be the network that has all the edges of both $G$ and $G'$. By symmetry, it suffices to show that the wired uniform spanning forests of $G$ and $G''$ have the same number of excessive ends a.s.  Let $\F$ and $\F''$ be samples of the WUSFs of $G$ and $G''$ respectively, and let $A$ be the set of edges of $G''$ that are not edges of $G$. Since $A$ is finite and $G$ is connected, the event $\sA = \{A \cap \F'' = \emptyset\}$ has positive probability (this implication is easily proven in several ways, e.g.\ using either Wilson's algorithm, the Aldous-Broder algorithm, or the Transfer Current Theorem \cite{BurPe93}). The spatial Markov property of the WUSF implies that the conditional distribution of $\F''$ given $\sA$ is equal to the distribution of $\F$, and in particular the conditional distribution of the number of excessive ends of $\F''$ given $\sA$ has the same distribution as the number of excessive ends of $\F$. The claim now follows from \cref{Thm:excessnotrandom}.
\end{proof}

\section{Ends and rough isometries}\label{Sec:Rough}
 Recall that a \textbf{rough isometry} from a graph $G=(V,E)$ to a graph $G'=(V',E')$ is a function $\phi:V\to V'$ such that, letting $d_G$ and $d_{G'}$ denote the graph distances on $V$ and $V'$, there exist positive constants $\alpha$ and $\beta$ such that the following conditions are satisfied:
\begin{enumerate}
\item \textbf{($\phi$ roughly preserves distances.)} For every pair of vertices $u,v\in V$, 
\[\alpha^{-1}d_G(u,v)-\beta \leq d_{G'}(\phi(u),\phi(v)) \leq \alpha d_G(u,v) +\beta.\]
\item \textbf{($\phi$ is almost surjective.)} For every vertex $v'\in V'$, there exists a vertex $v \in V$ such that $d_G(\phi(v),v')\leq \beta$.
\end{enumerate}
For background on rough isometries, see \cite[\S2.6]{LP:book}. 
The final result of this paper answers negatively Question 7.6 of Lyons, Morris and Schramm~\cite{LMS08}, which asked whether the property of having one-ended WUSF components is preserved under rough isometry of graphs.

\begin{thm}\label{thm:roughexample}
There exist two rough-isometric, bounded degree graphs $G$ and $G'$ such that every component of the wired uniform spanning forest of $G$ has one-end a.s., but the wired uniform spanning forest of $G'$ contains a component with uncountably many ends a.s. 
\end{thm}

The proof of \cref{thm:roughexample} uses Wilson's algorithm rooted at infinity. We refer the reader to \cite[Proposition 10.1]{LP:book} for an exposition of this algorithm. The description as a branching process of the past of the WUSF of a regular trees with height-dependent exponential edge stretching  is adapted from \cite[\S 11]{BLPS}, and first appeared in the work of H\"aggstr\"om \cite{haggstrom1998uniform}.

\begin{proof}[Proof of \cref{thm:roughexample}]
Let $T=(V,E)$ be a 3-regular tree with root $\rho$. We write $\Vert u \Vert$ for the distance between $u\in V$ and $\rho$. For each positive integer $k$, let $T_k=( V_k,E_k)$ denote the tree obtained from $T$ by replacing every edge connecting a vertex $u$ of $T$ to its parent by a path of length $k^{\Vert u\Vert }$. We identify the degree 3 vertices of $T_k$ with the vertices of $T$. 
For each vertex $u \in V$, let $S(u)$ be a binary tree  with root $\rho_u$ and let $S_k(u)$ be the tree obtained from $S(u)$ by replacing every edge with a path of length $k^{\Vert u\Vert +1}$. Finally, for each pair of positive integers $(k,m)$, let $G^m_k$ be the graph obtained from $T_k$ by, for each vertex $u \in V$, adding a path of length $k^{\Vert u\Vert +1}$ connecting $u$ to $\rho_u$ and then replacing every edge in each of these added paths and every edge in each of the trees $S_k(u)$ by $m$ parallel edges. The vertex degrees of $G^m_k$ are bounded by $3+m$, and the identity map is an isometry (and hence a rough isometry) between $G^m_k$ and $G^{m'}_k$ whenever $k,m$ and $m'$ are positive integers. See Figure 1 for an illustration.

\begin{figure}
\includegraphics[width=0.9\textwidth]{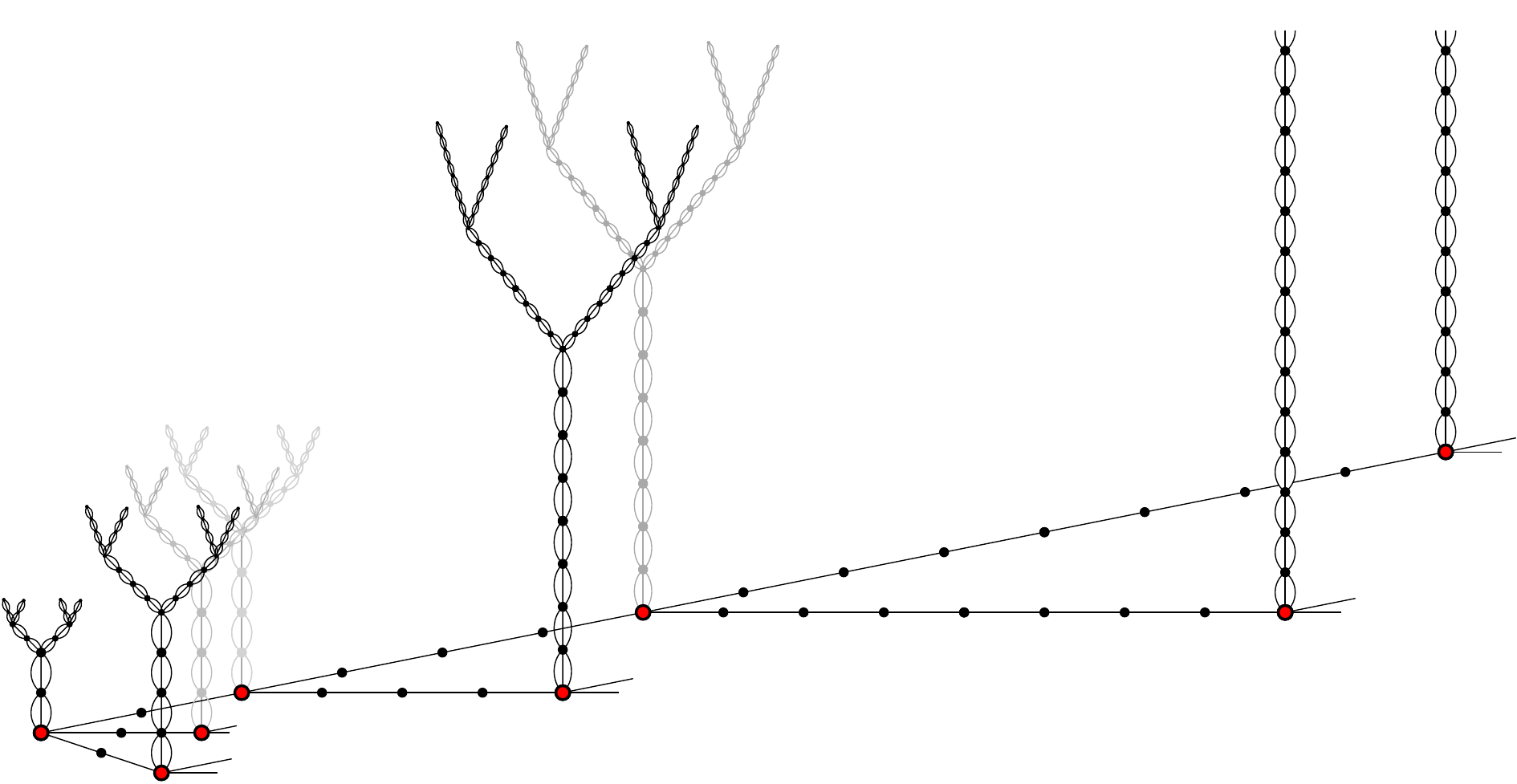}
\caption{An illustration of the graph $G_2^3$. Red vertices correspond to vertices of the $3$-regular tree $T$. Only three generations of each of the trees $S(v)$ are pictured.}
\end{figure}

Let $k$ and $m$ be positive integers. Observe that for every vertex $v$ of $T$ and every child $u$ of $v$ in $T$, the probability that simple random walk on $G_k^m$ started at $u$ ever hits $v$ does not depend on the choice of $v$ or $u$. 
Denote this probability $p(m,k)$. We can bound $p(m,k)$ as follows. 
\begin{equation}\label{eq:pbound}\frac{k}{k+2+m}\leq p(m,k)\leq \frac{k+2}{k+2+2m}. \end{equation}
The lower bound of $k/(k+2+m)$ is exactly the probability that the random walk started at $u$ visits $v$ before visiting any other vertex of $T$ or visiting $\rho_u$. The upper bound of $(k+2)/(k+2+2m)$ is exactly the probability that the random walk started at $u$ ever visits a neighbour of $u$ in $T$. This can be computed by a straightforward  network reduction (see \cite{LP:book} for background): The conductance to infinity from the root of a binary tree is $1$, so that, by the series and parallel laws, the  effective conductance to infinity from $u$ in the subgraph of $G^m_k$ spanned by the vertices of $S_k(u)$ and the path connecting $u$ to $S_k(u)$ is 
$2mk^{-\Vert u \Vert -1}$.
On the other hand, the effective conductance between $u$ and its parent $v$ is $k^{-\Vert u \Vert}$, while the effective conductance between $u$ and each of its children is $k^{-\Vert u \Vert -1}$. It follows that the probability that a random walk started at $u$ ever visits a neighbour of $u$ in $T$ is exactly
\[ \frac{ k^{-\Vert u \Vert} + 2k^{-\Vert u \Vert-1}}{k^{-\Vert u \Vert} + 2k^{-\Vert u \Vert-1}+2mk^{-\Vert u \Vert -1} }= \frac{k+2}{k+2+2m}\]
as claimed.

Let $\F^m_k$ be a sample of $\WUSF_{G_k^m}$ generated using Wilson's algorithm on $G_k^m$, starting with the root $\rho$ of $T$. Let $\xi$ be the loop-erased random walk in $G_k^m$ beginning at $\rho$ that is used to start our forest. The path $\xi$ includes either one or none of the neighbours of $\rho$ in $T$ and so, in either case, there are at least two neighbours $v_1$ and $v_2$ of $\rho$ in $T$ that are not contained in this path. Continuing to run Wilson's algorithm from $v_1$ and $v_2$, we see that, conditional on $\xi$, the events $A_1=\{v_1$ is in the past of $\rho$ in $\F_k^m\}$ and $A_2=\{v_2$ is in the past of $\rho$ in $\F_k^m\}$ are independent and each have probability $p(m,k)$. Furthermore, on the event $A_i$, we add only the path connecting $v_i$ and $\rho$ in $G_k^m$ to the forest during the corresponding step of Wilson's algorithm. Recursively, we see that the restriction to $T$ of the past of $\rho$ in $\F$ contains a Galton-Watson branching process with Binomial offspring distribution $(2, p(m,k))$. 
If $k \geq m + 3$  this branching process is supercritical, so that $\F_k^m$ contains a component with uncountably many ends with positive probability. By tail triviality of the WUSF \cite[Theorem 10.18]{LP:book}, $\F_k^m$ contains a component with uncountably many ends a.s.\ when $k\geq m +3$. 

On the other hand, a similar analysis shows that the restriction to $T$ of past of  $\rho$ in $\F_k^m$  is stochastically dominated by a binomial $(3,p(m,k))$ branching process. (The $3$ here is to account for the possibility that every child of $\rho$ in $T$ is in its past). If $m \geq k + 2$, this branching process is either critical or subcritical, and we conclude that the restriction to $T$ of the past of $\rho$ in $\F_k^m$ is finite a.s. Condition on this restriction. Similarly again to the above, the restriction to $S_k(v)$ of the past of $v$ in $\F_m^k$ is stochastically dominated by a critical binomial $(2,1/2)$ branching process for each vertex $v$ of $T$, and is therefore finite a.s. We conclude that the past of $\rho$ in $\F_k^m$ is finite a.s.\ whenever $m \geq k +2$. A similar analysis shows that the past in $\F_k^m$ of every vertex of $G_k^m$ is finite a.s., and consequently that every component of $\F_k^m$ is one-ended a.s.\ whenever $m \geq k +2$.

Since $4 \geq 1 +3$ and $6 \geq 4 +2$, the wired uniform spanning forest $\F_4^1$ of $G_4^1$ contains an infinitely-ended component a.s., and every component of the wired uniform spanning forest $\F_4^{6}$ of $G_4^{6}$ is one-ended a.s. \qedhere
\end{proof}

\section{Closing Discussion and Open Problems}\label{sec:problems}

\subsection{The FMSF of the interlacement ordering}
One way to think about the Interlacement Aldous-Broder algorithm is as follows. Given the interlacement process $\sI$ on a transient network $G$, we can define a total ordering of the edges of $G$ according to the order in which they are traversed by the trajectories of $\sI_{[0,\infty)}$. That is, we define a strict total ordering $\prec$ of $E$ by setting $e_1 \prec e_2$ if and only if either $e_1$ is first traversed by a trajectory of $\sI_{[0,\infty)}$ at a smaller time than $e_2$ is first traversed by a trajectory of $\sI_{[0,\infty)}$, or if $e_1$ and $e_2$ are both traversed for the first time by the same trajectory of $\sI_{[0,\infty)}$, and this trajectory traverses $e_1$ before it traverses $e_2$. We call $\prec$ the \textbf{interlacement ordering} of the edge set $E$.

It is easily verified that $\AB_0(\sI)$ is the \textbf{wired minimal spanning forest} of $G$ with respect to the interlacement ordering. That is, an edge $e \in E$ is included in $\AB_0(\sI)$ if and only if there does not exist either a finite cycle or a bi-infinite path in $G$ containing $e$ for which $e$ is the $\prec$-maximal element. See \cite{LP:book} for background on minimal spanning forests. In light of this, it is natural to wonder what might be said about the \emph{free} minimal spanning forest of the interlacement ordering, that is, the spanning forest of $G$ that includes an edge $e\in E$ if and only if there does not exist a finite cycle in $G$ containing $e$ for which $e$ is the $\prec$-maximal element. Indeed, if this forest were the FUSF of $G$, this could be used to solve  the monotone coupling problem \cite[Question 10.6]{LP:book} (see also \cite{lyons2016invariant,mester2013invariant,bowen2004couplings}) and the almost-connectivity problem \cite[Question 10.12]{LP:book}.

Unfortunately there is little reason for this to be the case other than wishful thinking.
Indeed,
let $\sI$ be the interlacement process on a transient network $G$, and define 
\[t_c = \inf \left\{ t \in (0,\infty) : \cI_{[0,t]} \text{ is connected a.s.}\right\}.\]
Teixeira and Tykesson \cite{teixeira2013random} proved that  if $G$ is transitive, then $t_c$ is positive if and only if $G$ is nonamenable. (The amenable case of their result generalises the corresponding result for $\Z^d$, due to Sznitman \cite{Szni10}.) 
 We can apply this result to prove that the free minimal spanning forest of the interlacement ordering is distinct from the WUSF on any nonamenable transitive graph: This is similar to how the usual FMSF and WMSF (where the edge weights are i.i.d.) are distinct if and only there is a nonempty nonuniqueness phase for Bernoulli bond percolation \cite{LPS06}. Since there are many nonamenable transitive graphs where the WUSF and FUSF coincide (e.g.\ the product of a $3$-regular tree with $\Z$, see \cite[Chapter 10]{LP:book}), we deduce that there are transitive graphs (indeed, Cayley graphs) for which the FUSF does not coincide with the free minimal spanning forest of the interlacement ordering. 

%

We now give a quick sketch of this argument. Suppose that $G$ is a transitive nonamenable graph.
Observe that for every $t<t_c$, there must exist a connected component of $\cI_{[0,t)}$ and a vertex $u$ of $G$ such that a random walk started at $u$ has a positive probability not to hit the component: If not, we would have that $\cI_{[0,s]}$ was connected for every $s>t$, contradicting the assumption that $t<t_c$. Moreover, by finding a path from $u$ to the component and considering the last vertex of the path before we reach the component, the vertex $u$ can be taken to be adjacent to the component.
Let $\tau$ be the first time after $t_c/2$ that $u$ is hit by a trajectory of $\sI$, and let $e_{t_c/2}(u)$ be the oriented edge that is traversed by this trajectory as it enters $u$ for the first time. Denote this trajectory by $W$. No other trajectories of $\sI$ appear at time $\tau$ a.s. 
In light of the above discussion, by making local modifications to finitely many trajectories in $\sI_{[0,\tau)}$, we see that the following event occurs with positive probability: $\tau$ is strictly less than $t_c$, the vertices $u$ and $e_{t_c/2}(u)^-$ are both in different components of $\cI_{[0,\tau)}$ (and, in particular, are both in $\cI_{[0,\tau)}$), and $W$ hits the component of $u$ in $\cI_{[0,\tau)}$ for the first time at $u$. On this event we must have that $e_{t_c/2}(u)$ is included in the free minimal spanning forest of the interlacement ordering, but is not in $\AB_0(\sI)$, and hence the two forests do not coincide.

\subsection{Exceptional times} A natural question raised by the Interlacement Aldous-Broder algorithm concerns the existence or non-existence of \textbf{exceptional times} for the process $\langle \F_t \rangle_{t\in \R} = \langle  \AB_t(\sI)\rangle_{t\in \R}$, that is, times at which $\F_t$ has properties markedly different from the a.s.\ properties of $\F_0$. For example, we might ask whether, considering the process $\langle \F_t \rangle_{t\in \R}$ on $\Z^d$ ($d\geq 3$), there are exceptional times  when  the forest has multiply ended components,  is disconnected (if $d=3,4$), or is connected (if $d\geq 5$). (Note that the proof of \cref{Thm:excessnotrandom} implies that there do not exist exceptional times at which $\F_t$ contains \emph{indestructible} excessive ends.)

The answers to the first of these questions turn out to rather simple. Given a trajectory $W$ in a graph $G$ and a vertex $u$ of $G$ visited by the path, we define $e(W,u)$ to be the oriented edge pointing into $u$ that is traversed by $W$ as it enters $u$ for the first time, and define
\[\AB(W) = \{-e(W,u) : u \text{ is visited by $W$}\}.\]
Note that if the trace of $W$ is infinite then $\AB(W)$ is an infinite oriented tree. We define the tree of first entry edges $\AB(X)$ similarly when $X=\langle X_n \rangle_{n\geq0}$ is a path in $G$.

\begin{prop}[Exceptional times for excessive ends]
\label{prop:exceptional}
Let $G$ be a transient network, let $\sI$ be the interlacement process, and let $\langle \F_t \rangle_{t\in \R} =\langle \AB_t(\sI) \rangle_{t\in \R}$. 
Let $\sE$ be the set of times $t\in \R$ such that $\F_t$ has a multiply ended component, and let $\sE'$ be the set of times $t\in \R$ for which there exists a trajectory $W_t$ in $\sI_t$ such that $\AB(W)$ is multiply ended. 
If every component of $\F_0$ is one-ended almost surely, then the following hold almost surely.
\begin{enumerate}
	\item 
 $\sE=\sE'$, and $\sE=\emptyset$ if and only if $\F_0$ is connected almost surely.
    \item
For every $t\in \sE$, there is exactly one two-ended component of $\F_t$, and all other components are one-ended. The unique two-ended component is the union of the tree $\AB(W_t)$ with some finite bushes. 
\end{enumerate}
\end{prop}

Since \cref{prop:exceptional} is tangential to the paper, we leave out some details from the proof.

\begin{proof}
We first prove that $\sE=\sE'$ almost surely. 
The containment $\sE \subseteq \sE'$ is immediate, and holds deterministically. 
Let $\Omega$ be the almost sure event that every component of $\F_t$ is one-ended for every rational $t$, and that no two trajectories of $\sI$ have the same arrival time. We claim that $\sE=\sE'$ pointwise on the event $\Omega$.
Suppose that 
$\Omega$ holds and that 
$t\in \sE$, so that there exists a sequence of vertices $\langle v_i \rangle_{i\geq0}$ such that $v_i = e_t(v_{i+1})^-$ for each $i\geq 0$. 
In particular, the arrival times $\tau_t(v_i)$ are increasing. 
We claim that we must have $\tau_t(v_i)=t$ for all $i\geq 0$. Indeed, if $\tau_t(v_i) \geq t+\eps$ for some $\eps>0$ and all $i$ larger than some $i_0$, then we would have that $e_{t+\delta}(v_i)=e_t(v_i)$ for all $0<\delta \leq \eps$ and all $i\geq i_0$. In this situation, we would therefore have that $\F_{t+\delta}$ contained a multiply ended component for every $0<\delta\leq \eps$, contradicting the assumption that the event $\Omega$ occured. Thus, we must have that there exists a trajectory $W_t \in \sI_t$ (which is unique by definition of $\Omega$), and the sequence $v_i$ gives an excessive end in the tree $\AB(W)$. Since the sequence $v_i$ represented an arbitrary excessive end of $\F_t$, it follows that every excessive end of $\F_t$ arises from the tree $\AB(W)$ on the event $\Omega$.

Now, if $\F_0$ is not connected a.s., then there is a vertex $v$ of $G$ such that two independent random walks from $v$ do not intersect with positive probability, and it follows that there a.s.\ exist trajectories in $\sI$ such that $\AB(W)$ has at least two ends. 
It remains to prove that the trees $\AB(W)$ have at most two-ends for every trajectory $W$ in $\sI$, and are all one-ended if $\F_0$ is a.s.\ connected. Since there are only countably many trajectories in $\sI$, it suffices to analyze a single bi-infinite random walk.
 To prove this, it is convenient to introduce a variant of the interlacement Aldous-Broder in which we first run a simple random walk started from a fixed vertex (considered to arrive at time zero), and then run the interlacement process $\sI_{[0,\infty)}$, and form a forest from the first entry edges. It is not difficult to see, by a slight modification of the proof \cref{thm:IAB}, that the forest produced this was is the wired uniform spanning forest: In the finite exhaustion, this corresponds to first running a random walk from $v$ until hitting the distinguished boundary vertex, and then decomposing the rest of the walk into excursions from the boundary vertex. 
  Using this algorithm, it follows that  $\AB(X)$ is one-ended a.s.\ whenever $X$ is a random walk on a transient graph $G$ for which the wired uniform spanning forest is one-ended.

 Now suppose that $W=\langle W_n \rangle_{n\in \Z}$ is a bi-infinite random walk. If $\langle v_i \rangle_{i\geq 0}$ is a sequence of vertices in $G$ corresponding to an excessive end of $\AB(W)$, then we must have that $v_i = e(v_{i+1})^-$ for all $i$ sufficiently large, and it follows that this excessive end must be an end of the tree $\AB(\langle W_i \rangle_{i \geq 0})$, completing the proof that $\AB(W)$ has at most two ends. On the other hand, we note that the unique path to infinity from $W_0$ in $\AB(\langle W_i \rangle_{i\geq0})$ is exactly the loop-erasure of $\langle W_i \rangle_{i\geq0}$, and if $\F_0$ is connected a.s.\ then this path is hit infinitely often a.s.\ by $\langle W_i \rangle_{i <0}$. We deduce that in this case this end is not present in the tree $\AB(W)$, completing the proof. 
\end{proof}

We do not know if there exist exceptional times for (dis)connectivity. We expect that such times do not exist, but it would be very interesting if they do.

\begin{question}
\label{Q:exceptionaltimes}
Let $d\geq 3$, let $\sI$ be the interlacement process on $\Z^d$, and let $\langle \F_t \rangle_{t\in \R}=\langle \AB_t(\sI)\rangle_{t\in \R}$. If $d=3,4$, do there exist times at which $\F_t$ is disconnected? If $d\geq 5$, do there exist times at which $\F_t$ is connected? 
\end{question}

If the answer to \cref{Q:exceptionaltimes} is positive, it would be interesting to further understand the structure of the set of exceptional times and the geometry of the forest $\F_t$ at a typical exceptional time.  It is easy to see that, unlike for excessive ends, the arrival times of trajectories are \emph{not} exceptional times for connectivity, so if exceptional times do exist they are likely to have a more interesting structure.  We note that there is a rich theory of exceptional times for other models such as dynamical percolation, addressing many analogous questions. See e.g.\ \cite{garban2010fourier,steif2009survey,MR2235173,hammond2015local}. 

A related question concerns the decorrelation of connectivity events under the dynamics.



\begin{question}\label{Q:correlations}
Let $d\geq 5$, let $\sI$ be the interlacement process on $\Z^d$, and let $\langle \F_t \rangle_{t\in \R}= \langle \AB_t(\sI) \rangle_{t\in \R}$. How does
\[\P(\text{$x$ is connected to $y$ in both $\F_0$ and $\F_t$})\]
behave as a function of the vertices $x,y \in \Z^d$ and the number $t>0$? Does the behaviour as a function of $x,y$ undergo a phase transition as $t$ is increased?
\end{question}

Recall that for the USF of $\Z^d$, $d\geq 5$, the probability that two vertices $x$ and $y$ are in the same component of the USF decays like $\|x-y\|^{-(d-4)}$ as $\|x-y\|\to\infty$ \cite{BeKePeSc04}.
%
%
%
A successful approach to \cref{Q:exceptionaltimes,Q:correlations}
 might need to draw more deeply on the interlacement literature than we have needed to in this paper.

\subsection{Excessive ends via update tolerance}

A key tool in the study of the USFs carried out in \cite{H15,HutNach15a,timar2015indistinguishability} is the \textbf{update-tolerance} of the USFs (referred to as \textbf{weak insertion tolerance} by Tim\'ar \cite{timar2015indistinguishability}). Given a sample $\F$ of either the WUSF and the FUSF of a network $G$ and an oriented edge $e$ of $G$ not in $\F$, update-tolerance states that there exists a forest $U(\F,e)$, obtained from $\F$ by adding $e$ and deleting some other appropriately chosen edge $d$, such that the law of $U(\F,e)$ is absolutely continuous with respect to that of $\F$. The forest $U(\F,e)$ is called the update of $\F$ at $e$. See \cite{H15,HutNach15a,timar2015indistinguishability} for further details. 

Since the number of excessive ends does not change when we perform an update, a positive solution of the following conjecture would yield an alternative proof of \cref{Thm:excessnotrandom}.
We say that a Borel set $\sA \subseteq \{0,1\}^E$ is \textbf{update-stable} if for every oriented edge $e$ of $G$, the updated forest $U(\F,e)$ is in $\sA$ if and only if $\F$ is in $\sA$ almost surely.  The conjecture would also imply a positive solution to \cite[Question 15.7]{BLPS}.

\begin{conjecture}
Let $G$ be an infinite network, and let $\F$ be either the wired or free spanning forest of $G$.  Then for every update-stable Borel set $\sA\subseteq \{0,1\}^E$, the probability that $\F$ is in $\sA$ is either zero or one. 
\end{conjecture}

\subsection{Ends in uniformly transient networks}
The following natural question remains open. If true, it would strengthen the results of Lyons, Morris and Schramm \cite{LMS08}.  A network is said to be \textbf{uniformly transient} if the capacities of the vertices of the network are bounded below by a positive constant. 

\begin{question}\label{Q:uniformtransience}
Let $G$ be a uniformly transient network with $\inf_e c(e)>0$. Does it follow that every component of the wired uniform spanning forest of $G$ is one-ended almost surely? 
\end{question}

The argument used in the proof of \cref{Thm:excessnotrandom} can be adapted to show that, under the hypotheses of \cref{Q:uniformtransience}, every component of the WUSF is either one-ended or has uncountably many ends, with no isolated excessive ends. To answer \cref{Q:uniformtransience} positively, it remains to rule this second case out.




\medskip

\subsection*{Acknowledgements}
This work was carried out while the author was an intern at Microsoft Research, Redmond. We thank Omer Angel, Ori Gurel-Gurevich, Ander Holroyd, Russ Lyons, Asaf Nachmias and Yuval Peres for useful discussions. We also thank Tyler Helmuth for his careful reading of an earlier version of this manuscript, and thank both Russ Lyons and the anonymous referee for suggesting many corrections and improvements to the initial preprint.

\footnotesize{
	\bibliographystyle{abbrv}
	\bibliography{unimodular}
 }
\end{document}

%% file: header.tex
\usepackage{amsmath,amssymb,amsfonts,amsthm}
\usepackage{color}

\usepackage[colorlinks=true,linkcolor=blue,citecolor=blue,urlcolor=black,pdfborder={0 0 0}]{hyperref}
\usepackage{graphicx}
\usepackage{cleveref}

\usepackage{caption}
\usepackage{thmtools}
\usepackage{commath}
\usepackage{mathrsfs}

\usepackage{enumitem}
\usepackage{eufrak}

\crefname{theorem}{Theorem}{Theorems}
\crefname{thm}{Theorem}{Theorems}
\crefname{lemma}{Lemma}{Lemmas}
\crefname{lem}{Lemma}{Lemmas}
\crefname{remark}{Remark}{Remarks}
\crefname{prop}{Proposition}{Propositions}
\crefname{defn}{Definition}{Definitions}
\crefname{corollary}{Corollary}{Corollaries}
\crefname{conjecture}{Conjecture}{Conjectures}
\crefname{question}{Question}{Questions}
\crefname{chapter}{Chapter}{Chapters}
\crefname{section}{Section}{Sections}
\crefname{figure}{Figure}{Figures}
\newcommand{\crefrangeconjunction}{--}

\theoremstyle{plain}
\newtheorem{thm}{Theorem}[section]

\newtheorem{lem}[thm]{Lemma}
\newtheorem{corollary}[thm]{Corollary}
\newtheorem{prop}[thm]{Proposition}
\newtheorem{conjecture}[thm]{Conjecture}

\newtheorem{question}[thm]{Question}
\theoremstyle{definition}
\newtheorem{defn}[thm]{Definition}

\theoremstyle{remark}

\numberwithin{equation}{section}

%% file: commands.tex
\renewcommand{\P}{\mathbb P}
\newcommand{\E}{\mathbb E}

\newcommand{\R}{\mathbb R}
\newcommand{\Z}{\mathbb Z}
\newcommand{\N}{\mathbb N}

\newcommand{\F}{\mathfrak F}

\newcommand{\cA}{\mathcal A}

\newcommand{\cG}{\mathcal G}

\newcommand{\cI}{\mathcal I}

\newcommand{\cW}{\mathcal W}

\newcommand{\sA}{\mathscr A}
\newcommand{\sB}{\mathscr B}

\newcommand{\sE}{\mathscr E}

\newcommand{\sI}{\mathscr I}

\newcommand{\WUSF}{\mathsf{WUSF}}

\newcommand{\UST}{\mathsf{UST}}
\newcommand{\OWUSF}{\mathsf{OWUSF}}


%% file: InterlacementsRevised.bbl
\begin{bibdiv}
\begin{biblist}

\bib{AL07}{article}{
      author={Aldous, David},
      author={Lyons, Russell},
       title={Processes on unimodular random networks},
        date={2007},
        ISSN={1083-6489},
     journal={Electron. J. Probab.},
      volume={12},
       pages={no. 54, 1454\ndash 1508},
         url={http://dx.doi.org/10.1214/EJP.v12-463},
      review={\MR{2354165 (2008m:60012)}},
}

\bib{Aldous90}{article}{
      author={Aldous, David~J.},
       title={The random walk construction of uniform spanning trees and
  uniform labelled trees},
        date={1990},
        ISSN={0895-4801},
     journal={SIAM J. Discrete Math.},
      volume={3},
      number={4},
       pages={450\ndash 465},
         url={http://dx.doi.org.ezproxy.library.ubc.ca/10.1137/0403039},
      review={\MR{1069105 (91h:60013)}},
}

\bib{barlow2014subsequential}{article}{
      author={Barlow, Martin~T},
      author={Croydon, David~A},
      author={Kumagai, Takashi},
       title={Subsequential scaling limits of simple random walk on the
  two-dimensional uniform spanning tree},
        date={2014},
     journal={arXiv preprint arXiv:1407.5162},
}

\bib{BeKePeSc04}{article}{
      author={Benjamini, Itai},
      author={Kesten, Harry},
      author={Peres, Yuval},
      author={Schramm, Oded},
       title={Geometry of the uniform spanning forest: transitions in
  dimensions {$4,8,12,\dots$}},
        date={2004},
        ISSN={0003-486X},
     journal={Ann. of Math. (2)},
      volume={160},
      number={2},
       pages={465\ndash 491},
  url={http://dx.doi.org.ezproxy.library.ubc.ca/10.4007/annals.2004.160.465},
      review={\MR{2123930 (2005k:60026)}},
}

\bib{BLPS}{article}{
      author={Benjamini, Itai},
      author={Lyons, Russell},
      author={Peres, Yuval},
      author={Schramm, Oded},
       title={Uniform spanning forests},
        date={2001},
        ISSN={0091-1798},
     journal={Ann. Probab.},
      volume={29},
      number={1},
       pages={1\ndash 65},
         url={http://dx.doi.org/10.1214/aop/1008956321},
      review={\MR{1825141 (2003a:60015)}},
}

\bib{bowen2004couplings}{article}{
      author={Bowen, Lewis},
       title={Couplings of uniform spanning forests},
        date={2004},
     journal={Proceedings of the American Mathematical Society},
       pages={2151\ndash 2158},
}

\bib{broder1989generating}{inproceedings}{
      author={Broder, Andrei},
       title={Generating random spanning trees},
organization={IEEE},
        date={1989},
   booktitle={Foundations of computer science, 1989., 30th annual symposium
  on},
       pages={442\ndash 447},
}

\bib{BurPe93}{article}{
      author={Burton, Robert},
      author={Pemantle, Robin},
       title={Local characteristics, entropy and limit theorems for spanning
  trees and domino tilings via transfer-impedances},
        date={1993},
        ISSN={0091-1798},
     journal={Ann. Probab.},
      volume={21},
      number={3},
       pages={1329\ndash 1371},
  url={http://links.jstor.org.ezproxy.library.ubc.ca/sici?sici=0091-1798(199307)21:3<1329:LCEALT>2.0.CO;2-L&origin=MSN},
      review={\MR{1235419 (94m:60019)}},
}

\bib{CerTei12}{book}{
      author={{\v{C}}ern{\'y}, Ji{\v{r}}{\'{\i}}},
      author={Teixeira, Augusto~Quadros},
       title={From random walk trajectories to random interlacements},
      series={Ensaios Matem\'aticos [Mathematical Surveys]},
   publisher={Sociedade Brasileira de Matem\'atica, Rio de Janeiro},
        date={2012},
      volume={23},
        ISBN={978-85-85818-69-2},
      review={\MR{3014964}},
}

\bib{chen2004anchored}{article}{
      author={Chen, Dayue},
      author={Peres, Yuval},
      author={Pete, Gabor},
       title={Anchored expansion, percolation and speed},
        date={2004},
     journal={Annals of probability},
       pages={2978\ndash 2995},
}

\bib{Dhar90}{article}{
      author={Dhar, Deepak},
       title={Self-organized critical state of sandpile automaton models},
        date={1990},
        ISSN={0031-9007},
     journal={Phys. Rev. Lett.},
      volume={64},
      number={14},
       pages={1613\ndash 1616},
  url={http://dx.doi.org.ezproxy.library.ubc.ca/10.1103/PhysRevLett.64.1613},
      review={\MR{1044086 (90m:82053)}},
}

\bib{DrRaBa14}{book}{
      author={Drewitz, Alexander},
      author={R{\'a}th, Bal{\'a}zs},
      author={Sapozhnikov, Art{\"e}m},
       title={An introduction to random interlacements},
      series={Springer Briefs in Mathematics},
   publisher={Springer, Cham},
        date={2014},
        ISBN={978-3-319-05851-1; 978-3-319-05852-8},
  url={http://dx.doi.org.ezproxy.library.ubc.ca/10.1007/978-3-319-05852-8},
      review={\MR{3308116}},
}

\bib{garban2010fourier}{article}{
      author={Garban, Christophe},
      author={Pete, G{\'a}bor},
      author={Schramm, Oded},
       title={The fourier spectrum of critical percolation},
        date={2010},
     journal={Acta Mathematica},
      volume={205},
      number={1},
       pages={19\ndash 104},
}

\bib{haggstrom1998uniform}{article}{
      author={H{\"a}ggstr{\"o}m, Olle},
       title={Uniform and minimal essential spanning forests on trees},
        date={1998},
     journal={Random Structures \& Algorithms},
      volume={12},
      number={1},
       pages={27\ndash 50},
}

\bib{hammond2015local}{article}{
      author={Hammond, Alan},
      author={Pete, G{\'a}bor},
      author={Schramm, Oded},
      author={others},
       title={Local time on the exceptional set of dynamical percolation and
  the incipient infinite cluster},
        date={2015},
     journal={The Annals of Probability},
      volume={43},
      number={6},
       pages={2949\ndash 3005},
}

\bib{MR2235173}{article}{
      author={Hoffman, Christopher},
       title={Recurrence of simple random walk on {${\Bbb Z}^2$} is dynamically
  sensitive},
        date={2006},
        ISSN={1980-0436},
     journal={ALEA Lat. Am. J. Probab. Math. Stat.},
      volume={1},
       pages={35\ndash 45},
      review={\MR{2235173}},
}

\bib{H15}{article}{
      author={Hutchcroft, Tom},
       title={Wired cycle-breaking dynamics for uniform spanning forests},
     journal={Ann. Probab.},
        note={to appear},
}

\bib{HutNach15a}{article}{
      author={Hutchcroft, Tom},
      author={Nachmias, Asaf},
       title={Indistinguishability of trees in uniform spanning forests},
        date={2016},
        ISSN={1432-2064},
     journal={Probability Theory and Related Fields},
       pages={1\ndash 40},
         url={http://dx.doi.org/10.1007/s00440-016-0707-3},
}

\bib{Jar14}{unpublished}{
      author={J{\'a}rai, Antal~A.},
       title={Sandpile models},
         url={http://arxiv.org/abs/1401.0354v2},
        note={Preprint (2014)},
}

\bib{JarRed08}{article}{
      author={J{\'a}rai, Antal~A.},
      author={Redig, Frank},
       title={Infinite volume limit of the abelian sandpile model in dimensions
  {$d\geq 3$}},
        date={2008},
        ISSN={0178-8051},
     journal={Probab. Theory Related Fields},
      volume={141},
      number={1-2},
       pages={181\ndash 212},
  url={http://dx.doi.org.ezproxy.library.ubc.ca/10.1007/s00440-007-0083-0},
      review={\MR{2372969 (2009c:60268)}},
}

\bib{JarWer14}{article}{
      author={J{\'a}rai, Antal~A.},
      author={Werning, Nicol{\'a}s},
       title={Minimal configurations and sandpile measures},
        date={2014},
        ISSN={0894-9840},
     journal={J. Theoret. Probab.},
      volume={27},
      number={1},
       pages={153\ndash 167},
  url={http://dx.doi.org.ezproxy.library.ubc.ca/10.1007/s10959-012-0446-z},
      review={\MR{3174221}},
}

\bib{Ken00}{article}{
      author={Kenyon, Richard},
       title={The asymptotic determinant of the discrete {L}aplacian},
        date={2000},
        ISSN={0001-5962},
     journal={Acta Math.},
      volume={185},
      number={2},
       pages={239\ndash 286},
         url={http://dx.doi.org.ezproxy.library.ubc.ca/10.1007/BF02392811},
      review={\MR{1819995 (2002g:82019)}},
}

\bib{Klenkebook}{book}{
      author={Klenke, Achim},
       title={Probability theory},
      series={Universitext},
   publisher={Springer-Verlag London, Ltd., London},
        date={2008},
        ISBN={978-1-84800-047-6},
  url={http://dx.doi.org.ezproxy.library.ubc.ca/10.1007/978-1-84800-048-3},
        note={A comprehensive course, Translated from the 2006 German
  original},
      review={\MR{2372119 (2009f:60001)}},
}

\bib{Lawler80}{article}{
      author={Lawler, Gregory~F.},
       title={A self-avoiding random walk},
        date={1980},
        ISSN={0012-7094},
     journal={Duke Math. J.},
      volume={47},
      number={3},
       pages={655\ndash 693},
  url={http://projecteuclid.org.ezproxy.library.ubc.ca/euclid.dmj/1077314188},
      review={\MR{587173 (81j:60081)}},
}

\bib{LaSchWe04}{article}{
      author={Lawler, Gregory~F.},
      author={Schramm, Oded},
      author={Werner, Wendelin},
       title={Conformal invariance of planar loop-erased random walks and
  uniform spanning trees},
        date={2004},
        ISSN={0091-1798},
     journal={Ann. Probab.},
      volume={32},
      number={1B},
       pages={939\ndash 995},
         url={http://dx.doi.org.ezproxy.library.ubc.ca/10.1214/aop/1079021469},
      review={\MR{2044671 (2005f:82043)}},
}

\bib{LP:book}{book}{
      author={Lyons, R.},
      author={Peres, Y.},
       title={Probability on trees and networks},
   publisher={Cambridge University Press},
        date={2015},
        note={In preparation. Current version available at \hfil\break {\tt
  http://mypage.iu.edu/\string~rdlyons/}},
}

\bib{LMS08}{article}{
      author={Lyons, Russell},
      author={Morris, Benjamin~J.},
      author={Schramm, Oded},
       title={Ends in uniform spanning forests},
        date={2008},
        ISSN={1083-6489},
     journal={Electron. J. Probab.},
      volume={13},
       pages={no. 58, 1702\ndash 1725},
         url={http://dx.doi.org.ezproxy.library.ubc.ca/10.1214/EJP.v13-566},
      review={\MR{2448128 (2010a:60031)}},
}

\bib{LPS06}{article}{
      author={Lyons, Russell},
      author={Peres, Yuval},
      author={Schramm, Oded},
       title={Minimal spanning forests},
        date={2006},
        ISSN={0091-1798},
     journal={Ann. Probab.},
      volume={34},
      number={5},
       pages={1665\ndash 1692},
         url={http://dx.doi.org/10.1214/009117906000000269},
      review={\MR{2271476 (2007m:60299)}},
}

\bib{lyons2016invariant}{article}{
      author={Lyons, Russell},
      author={Thom, Andreas},
       title={Invariant coupling of determinantal measures on sofic groups},
        date={2016},
     journal={Ergodic Theory and Dynamical Systems},
      volume={36},
      number={02},
       pages={574\ndash 607},
}

\bib{mester2013invariant}{article}{
      author={Mester, P{\'e}ter},
      author={others},
       title={Invariant monotone coupling need not exist},
        date={2013},
     journal={The Annals of Probability},
      volume={41},
      number={3A},
       pages={1180\ndash 1190},
}

\bib{Morris03}{article}{
      author={Morris, Ben},
       title={The components of the wired spanning forest are recurrent},
        date={2003},
        ISSN={0178-8051},
     journal={Probab. Theory Related Fields},
      volume={125},
      number={2},
       pages={259\ndash 265},
  url={http://dx.doi.org.ezproxy.library.ubc.ca/10.1007/s00440-002-0236-0},
      review={\MR{1961344 (2004a:60024)}},
}

\bib{Pem91}{article}{
      author={Pemantle, Robin},
       title={Choosing a spanning tree for the integer lattice uniformly},
        date={1991},
        ISSN={0091-1798},
     journal={Ann. Probab.},
      volume={19},
      number={4},
       pages={1559\ndash 1574},
  url={http://links.jstor.org.ezproxy.library.ubc.ca/sici?sici=0091-1798(199110)19:4<1559:CASTFT>2.0.CO;2-E&origin=MSN},
      review={\MR{1127715 (92g:60014)}},
}

\bib{PeresClimb}{incollection}{
      author={Peres, Yuval},
       title={Probability on trees: an introductory climb},
        date={1999},
   booktitle={Lectures on probability theory and statistics ({S}aint-{F}lour,
  1997)},
      series={Lecture Notes in Math.},
      volume={1717},
   publisher={Springer, Berlin},
       pages={193\ndash 280},
  url={http://dx.doi.org.ezproxy.library.ubc.ca/10.1007/978-3-540-48115-7_3},
      review={\MR{1746302 (2001c:60139)}},
}

\bib{peres2004scaling}{article}{
      author={Peres, Yuval},
      author={Revelle, David},
       title={Scaling limits of the uniform spanning tree and loop-erased
  random walk on finite graphs},
        date={2004},
     journal={arXiv preprint math},
        ISSN={0410430/},
}

\bib{Pete08}{article}{
      author={Pete, G{\'a}bor},
       title={A note on percolation on {$\Bbb Z^d$}: isoperimetric profile via
  exponential cluster repulsion},
        date={2008},
        ISSN={1083-589X},
     journal={Electron. Commun. Probab.},
      volume={13},
       pages={377\ndash 392},
         url={http://dx.doi.org.ezproxy.library.ubc.ca/10.1214/ECP.v13-1390},
      review={\MR{2415145 (2010b:60272)}},
}

\bib{Schramm00}{article}{
      author={Schramm, Oded},
       title={Scaling limits of loop-erased random walks and uniform spanning
  trees},
        date={2000},
        ISSN={0021-2172},
     journal={Israel J. Math.},
      volume={118},
       pages={221\ndash 288},
         url={http://dx.doi.org.ezproxy.library.ubc.ca/10.1007/BF02803524},
      review={\MR{1776084 (2001m:60227)}},
}

\bib{schweinsberg2009loop}{article}{
      author={Schweinsberg, Jason},
       title={The loop-erased random walk and the uniform spanning tree on the
  four-dimensional discrete torus},
        date={2009},
     journal={Probability Theory and Related Fields},
      volume={144},
      number={3-4},
       pages={319\ndash 370},
}

\bib{steif2009survey}{incollection}{
      author={Steif, Jeffrey~E},
       title={A survey of dynamical percolation},
        date={2009},
   booktitle={Fractal geometry and stochastics iv},
   publisher={Springer},
       pages={145\ndash 174},
}

\bib{Szni10}{article}{
      author={Sznitman, Alain-Sol},
       title={Vacant set of random interlacements and percolation},
        date={2010},
        ISSN={0003-486X},
     journal={Ann. of Math. (2)},
      volume={171},
      number={3},
       pages={2039\ndash 2087},
  url={http://dx.doi.org.ezproxy.library.ubc.ca/10.4007/annals.2010.171.2039},
      review={\MR{2680403 (2011g:60185)}},
}

\bib{Szni12book}{book}{
      author={Sznitman, Alain-Sol},
       title={Topics in occupation times and {G}aussian free fields},
      series={Zurich Lectures in Advanced Mathematics},
   publisher={European Mathematical Society (EMS), Z\"urich},
        date={2012},
        ISBN={978-3-03719-109-5},
         url={http://dx.doi.org.ezproxy.library.ubc.ca/10.4171/109},
      review={\MR{2932978}},
}

\bib{Teix09}{article}{
      author={Teixeira, A.},
       title={Interlacement percolation on transient weighted graphs},
        date={2009},
        ISSN={1083-6489},
     journal={Electron. J. Probab.},
      volume={14},
       pages={no. 54, 1604\ndash 1628},
         url={http://dx.doi.org.ezproxy.library.ubc.ca/10.1214/EJP.v14-670},
      review={\MR{2525105 (2011b:60393)}},
}

\bib{teixeira2013random}{article}{
      author={Teixeira, Augusto},
      author={Tykesson, Johan},
      author={others},
       title={Random interlacements and amenability},
        date={2013},
     journal={The Annals of Applied Probability},
      volume={23},
      number={3},
       pages={923\ndash 956},
}

\bib{timar2015indistinguishability}{article}{
      author={Tim{\'a}r, {\'A}d{\'a}m},
       title={Indistinguishability of components of random spanning forests},
        date={2015},
     journal={arXiv preprint arXiv:1506.01370},
}

\bib{Wilson96}{inproceedings}{
      author={Wilson, David~Bruce},
       title={Generating random spanning trees more quickly than the cover
  time},
        date={1996},
   booktitle={Proceedings of the {T}wenty-eighth {A}nnual {ACM} {S}ymposium on
  the {T}heory of {C}omputing ({P}hiladelphia, {PA}, 1996)},
   publisher={ACM, New York},
       pages={296\ndash 303},
         url={http://dx.doi.org/10.1145/237814.237880},
      review={\MR{1427525}},
}

\end{biblist}
\end{bibdiv}
